\documentclass[12pt]{amsart}
\usepackage{amsmath}
\usepackage{amssymb}
\usepackage{mathtools,nccmath,textcomp}
\usepackage{tikz}
\usepackage[all]{xy}
\usepackage{longtable}
\usepackage{url}
\usepackage{textcomp}
\usepackage{booktabs}
\usepackage{multirow}
\usepackage{array}
\usepackage[normalem]{ulem} % tracked-changes copy
\allowdisplaybreaks[1]
\makeatletter
\@namedef{subjclassname@2020}{%
  \textup{2020} Mathematics Subject Classification}
\makeatother

\newtheorem{theorem}{Theorem}
\newtheorem{lemma}[theorem]{Lemma}
\newtheorem{proposition}[theorem]{Proposition}
\newtheorem{corollary}[theorem]{Corollary}
\newtheorem{conjecture}[theorem]{Conjecture}

\theoremstyle{definition}

\newtheorem{example}[theorem]{Example}

\theoremstyle{remark}
\newtheorem{remark}[theorem]{Remark}

\newtheorem{problem}[theorem]{Problem}
\numberwithin{equation}{section}

\DeclareMathAlphabet{\matheur}{U}{eur}{m}{n}

\newcommand{\FF}{\mathbb{F}}
\newcommand{\ZZ}{\mathbb{Z}}
\newcommand{\QQ}{\mathbb{Q}}

\newcommand{\PP}{\mathbb{P}}
\newcommand{\NN}{\mathbb{N}}

\newcommand{\rad}{\mathrm{rad}}
\newcommand{\Et}{E^{\mathrm{tor}}(\QQ)}

\DeclareMathOperator{\Aut}{Aut}

\DeclareMathOperator{\Ker}{Ker}
\DeclareMathOperator{\GL}{GL}

\DeclareMathOperator{\Gal}{Gal}

\DeclareMathOperator{\Cl}{Cl}

\newcommand{\Qbar}{\overline{\QQ}}

\newmuskip\pFqskip
\mathchardef\pFcomma=\mathcode`, % keep a copy of the comma

\begin{document}
%~
\title[Arithmetic exceptionality of Latt\`{e}s maps]{Arithmetic exceptionality of Latt\`{e}s maps}

\author{Songpon Sriwongsa}
\address{Department of Mathematics, Faculty of Science, King Mongkut's University of Technology Thonburi (KMUTT), Bangkok 10140, Thailand} \email{songpon.sri@kmutt.ac.th}

\author{Detchat Samart}
\address{Department of Mathematics, Faculty of Science, Burapha University, Chonburi 20131, Thailand} 
\email{petesamart@gmail.com}

\author{Chatchawan Panraksa}
\address{Applied Mathematics Program, Mahidol University International College, Nakhon Pathom 73170, Thailand} \email{chatchawan.pan@mahidol.edu}

%\thanks{}

\date{\today}

\begin{abstract}
Let $\FF_q$ denote a finite field of order $q$. A rational function $r(x)\in \QQ(x)$ is said to be \emph{arithmetically exceptional} if it induces a permutation on $\PP^1(\FF_p)$ for infinitely many primes $p$. Based on some computational results, Odaba\c{s} conjectured that for each $k\in \NN$, the $k$-th Latt\`{e}s map attached to an elliptic curve $E/\QQ$ is arithmetically exceptional if and only if $E$ has no $k$-torsion point whose $x$-coordinate is rational. In this paper, we prove that this conjecture is true for any elliptic curve $E/\QQ$ having complex multiplication (CM) by an imaginary quadratic field other than $\QQ(\sqrt{-11}).$ On the other hand, we show that the conjecture becomes invalid if $E$ has CM by $\QQ(\sqrt{-11})$ and $6\mid k$. Partial results for non-CM elliptic curves are also given.
\end{abstract}

\maketitle

\section{Introduction} \label{S:intro}

The study of polynomials and rational functions that induce permutations on finite fields is a classical topic at the intersection of number theory and arithmetic geometry. A polynomial $f \in \mathbb{Z}[x]$ is called a \emph{Schur polynomial} (or is said to be exceptional) if it induces a bijection on the residue field $\mathbb{Z}/p\mathbb{Z}$ for infinitely many primes $p$. This direction of research was initiated by Schur \cite{Schur1923Uber} in 1923. His work led to a conjecture, now known as \emph{Schur's conjecture}, asserting that any such polynomial must be a composition of linear polynomials and Dickson polynomials. Schur's conjecture was first proven by Fried \cite{Fried1970}, but his proof contains several gaps and errors. Later, Turnwald \cite{Turnwald1995} gave a corrected proof together with some historical remarks and M\"{u}ller \cite{Muller1997} gave a simplified proof, which relies on group-theoretic arguments.

The problem naturally extends to rational functions. A rational function $r(x) \in \mathbb{Q}(x)$ is said to be \emph{arithmetically exceptional} if it induces a permutation on the projective line $\mathbb{P}^1(\mathbb{F}_p)$ for infinitely many primes $p$. The classification of such functions is significantly more complex than the polynomial case. A rich source of rational functions with interesting dynamical and arithmetic properties are the \emph{Latt\`es maps}. Originating from the arithmetic of elliptic curves, given $k\in \NN$ and an elliptic curve $E/K$, the Latt\`es map $L_k$ attached to $E$ is the unique rational function satisfying $\mathrm{proj} \circ [k] = L_k \circ \mathrm{proj}$, where $[k]$ is the multiplication-by-$k$ endomorphism on $E$, and $\mathrm{proj}: E \to \mathbb{P}^1(K)$ is the $x$-coordinate projection map.  In the context of complex dynamics, Milnor \cite{Milnor2006} provided a foundational description of these maps, characterizing them as quotients of affine maps on tori. 

The permutation properties of Lattès maps over finite fields has been investigated extensively over the past few decades. We briefly review the main developments in chronological order. In \cite{GMS2003}, Guralnick, M\"uller, and Saxl provide a comprehensive analysis of arithmetically exceptional rational functions, highlighting their connection to monodromy groups and arithmetic geometry. In their framework, arithmetic exceptionality is governed by a purely group-theoretic condition on the pair of arithmetic and geometric monodromy groups, and their classification of \emph{indecomposable} arithmetically exceptional rational functions shows that all rational functions beyond the R\'edei and Dickson families and finitely many sporadic examples arise from isogenies or endomorphisms of elliptic curves. 
For example, rational functions of signature $(2,2,2,2),(2,3,6),(3,3,3),$ and $(2,4,4)$ (see \cite[\S 1]{GMS2003} for the precise definition) are known to be of Latt\'{e}s type. For each odd prime $p$, the multiplication-by-$p$ endomorphism on a suitable elliptic curve yields an indecomposable arithmetically exceptional rational map of degree $p^2$ and signature $(2,2,2,2)$ \cite[Thms.~1.4(b), ~6.5, and ~6.16]{GMS2003}. Similar results are also obtained for the remaining signatures. Note, however, that their results do not cover the entire family of Latt\'{e}s maps, since $L_k$ is decomposable if $k$ is composite, and compositions of arithmetically exceptional functions need not remain arithmetically exceptional \cite[Thm.~1.6]{GMS2003}.  Later, Küçüksakall\i \ \cite{Kucuksakalli2014} determined the value sets of Latt\`es maps induced by isogenies of elliptic curves with complex multiplication (CM) over a \emph{fixed} finite field $\FF_{q}$, allowing endomorphism rings in arbitrary imaginary quadratic fields. His results include a criterion for such a map to permute $\PP^1(\FF_{q})$ \cite[Cor.~2.8]{Kucuksakalli2014}. Likewise working over a fixed finite field, and motivated by cryptographic applications, Bisson and Tibouchi \cite{BT2018} constructed permutation rational functions on $\FF_q$ from isogenies of elliptic curves defined over $\FF_q$, suggesting a deep link between the endomorphism structure of the underlying curve and the permutation properties of the induced rational map. In addition to the aforementioned results, Bell et al.\ \cite{BCCHLTZ22} investigated the density of periodic points of Latt\`es maps over finite fields. As a by-product, they obtained a permutation criterion for these maps, which can be seen as an extension of Küçüksakall\i's result. Closest in spirit to our work, K\"u\c{c}\"uksakall{\i} and \"Onsiper \cite{KO2018} studied arithmetic exceptionality for generalized Latt\`es maps on $\PP^2$ arising from the surface $E\times E$. Using supersingular primes, they proved that every CM elliptic curve over a number field gives rise to at least one arithmetically exceptional map in the associated family.

The question addressed in the present paper is of a different, more global nature. For a fixed elliptic curve $E/\QQ$, we ask, for each positive integer $k$, whether $L_k$ induces a permutation on $\PP^1(\FF_p)$ for infinitely many primes $p$, rather than studying the behavior of $L_k$ over a single finite field. Recently, based on computational evidence, Odaba{\c{s}} \cite[Ch.~5]{Odabas2023} formulated a precise conjecture relating the arithmetic exceptionality of $L_k$ to the torsion structure of the underlying elliptic curve. In that work, elliptic curves of the form $y^2=x^3+Ax+B$ with integer parameters satisfying $-3\le A,B\le 3$ are investigated, and the densities of permutation primes of $L_k$ are tabulated for $k\in\{2,3,5,7\}$ over the first $1229$ primes (i.e., $p<10^4$). The conjecture can be stated as follows.

\begin{conjecture}[Odaba{\c{s}}]\label{C:Obadas}
For any positive integer $k$, the $k$-th Latt\`es map attached to an elliptic curve $E/\mathbb{Q}$ is arithmetically exceptional if and only if $E$ possesses no $k$-torsion point whose $x$-coordinate is rational.
\end{conjecture}

We investigate this conjecture primarily for the class of elliptic curves with complex multiplication.
By leveraging the theory of Frobenius distributions for CM elliptic curves and explicit reciprocity laws, we prove the following:

\begin{theorem}\label{T:main}
Let $K$ be an imaginary quadratic field with class number 1, excluding $K = \mathbb{Q}(\sqrt{-11})$, and let $E/\mathbb{Q}$ be an elliptic curve with CM by an order $\mathcal{O}$ in $K$. Then Conjecture~\ref{C:Obadas} holds for $E$.
\end{theorem}
The CM field $\QQ(\sqrt{-11})$ is excluded from Theorem~\ref{T:main}, since, as shown in Section~\ref{S:11}, the conjecture turns out to be false in this case. 
In particular, for $K=\QQ(\sqrt{-11})$, we show that the conjecture holds precisely when $6\nmid k$
(see Proposition~\ref{P:D11_6notdivk_exc}),
whereas for $6\mid k$ there is a genuine obstruction coming from the splitting of $3$ in $K$.
In fact, for any $E/\QQ$ with CM by the ring of integers $\mathcal{O}_K$ and any $k$ divisible by $6$,
the associated Latt\`es map $L_k$ fails to be arithmetically exceptional, regardless of the rationality of the $x$-coordinates of the $k$-torsion points of $E$
(see Proposition~\ref{P:D11_6divk_notexc}). 

It should be remarked that the hypothesis of Theorem~\ref{T:main} is less restrictive than it may appear. In fact, an elliptic curve over $\Qbar$ with CM by an order $\mathcal{O}$ in an imaginary quadratic field admits a model over $\QQ$ if and only if $j(\mathcal{O})\in \QQ$, which happens if and only if $\mathcal{O}$ has class number one (see, e.g., \cite{Cohen1993}). There are precisely thirteen such orders, of discriminants
\[
D\in\{-3,-4,-7,-8,-11,-12,-16,-19,-27,-28,-43,-67,-163\}.
\]
Consequently, Theorem~\ref{T:main}, together with the analysis of the remaining field $\QQ(\sqrt{-11})$ in Section~\ref{S:11}, covers all elliptic curves over $\QQ$ with complex multiplication.
We also give explicit counterexamples, for both CM and non-CM elliptic curves, in Proposition~\ref{P:D11_counterexample_curve} and Proposition~\ref{P:counter_noncm}, which appear to be overlooked in Odaba\c{s}'s experiments.
\medskip

%We also discuss the distribution of Frobenius traces in short intervals \cite{ABM2017} where relevant to the density of primes, and we analyze why the case of discriminant $-11$ presents unique obstructions.

%Organization
We organize this paper as follows.
Section~\ref{S:aux} collects the permutation criterion for Latt\`es maps in terms of Frobenius traces,
together with CM background results and reciprocity tools used in the proofs.
In Section~\ref{S:proof} we prove Theorem~\ref{T:main} by treating the CM discriminants case-by-case.
Section~\ref{S:11} is devoted to the exceptional field $K=\QQ(\sqrt{-11})$, where we identify
the obstruction when $6\mid k$ and prove the corresponding positive result for $6\nmid k$.
In Section~\ref{S:examples}, we present explicit examples illustrating the main theorem and
counterexamples for the CM field $\QQ(\sqrt{-11})$. Finally, in Section~\ref{S:outlook} we discuss the conjecture in a general setting, propose a corrected, isogeny-invariant reformulation, and outline directions for further investigation.

\section{Some auxiliary results}\label{S:aux}
In this section, we summarize some key results that play a crucial role in the proofs of our main results. We first briefly recall the bijectivity criterion for reduced Latt\'{e}s maps over finite fields. Let $K$ be an imaginary quadratic field with Hilbert class field $H$ and let $E$ be an elliptic curve defined over $H$ having complex multiplication by the ring of integers $\mathcal{O}_K$. For each $\alpha\in \mathcal{O}_K$, the \emph{Latt\'{e}s map} $F_\alpha$ is defined as a rational function with coefficients in $H$ satisfying 
\[F_\alpha(\wp(z))=\wp(\alpha z),\]
where $\wp(z)$ is the Weierstrass $\wp$-function associated to $E$. Hence for any positive integer $k$ the map $F_k$ coincides with the $k$-Latt\'{e}s map $L_k$ defined in Section~\ref{S:intro}. 
In \cite{Kucuksakalli2014}, K\"{u}\c{c}\"{u}ksakall{\i} studies the value sets of the reduced map $\bar{F}_\alpha:\PP^1(\FF_{p^m})\rightarrow \PP^1(\FF_{p^m})$ and proves that $\bar{F}_\alpha$ is bijective if and only if $\alpha$ is coprime to $\pi^{2m}-1$, where $\pi\in \mathcal{O}_K$ is a certain prime with $N(\pi)=p.$ More recently, Bell et al. \cite{BCCHLTZ22} established an analogous result without assuming complex multiplication.
\begin{proposition}{\cite[Cor.~2.6]{BCCHLTZ22}}\label{P:perm}
Let $E/\QQ$ be an elliptic curve which has a good reduction at a prime $p$ with the Frobenius trace $a_p$. Then the $k$-th  Latt\`{e}s map attached to $E$ induces a bijection on $\PP^1(\FF_p)$ if and only if $((p+1)^2-a_p^2,k)=1$.
\end{proposition}
We will also invoke the following well-known result of Deuring on Frobenius traces of CM elliptic curves. For references, the reader can consult \cite[Thm.~7.2.15]{Cohen1993}.
\begin{proposition}\label{P:trace}
If $E/\QQ$ has CM by an imaginary quadratic order $\mathcal{O}$ of discriminant $D$ and $p$ is a prime of good reduction and not dividing $D$, then 
\[a_p=\begin{cases}
    0 & \text{if} \left(\frac{D}{p}\right)=-1,\\
    \pi +\bar{\pi} &  \text{if} \left(\frac{D}{p}\right)=1,
\end{cases}
\]
for a suitable choice of prime $\pi\in \mathcal{O}$, depending on $E$, such that $N(\pi)=p$.
\end{proposition}
It can be seen immediately from Proposition~\ref{P:trace} that if $\left(\frac{D}{p}\right)=1$ (i.e. $p$ splits in $\QQ(\sqrt{D})$), then 
\begin{equation}\label{E:Np}
    (p+1)^2-a_p^2=N((\pi-1)(\pi+1))=|E(\FF_p)||E'(\FF_p)|,
\end{equation} 
where $E'$ is a quadratic twist of $E$ (see, e.g., \cite[Lem.~2.3]{BCCHLTZ22}). For split primes $p$, we will use the quantities appearing in \eqref{E:Np} interchangeably in the subsequent sections. The next lemma gives us a necessary condition for the divisibility of $N((\pi-1)(\pi+1))$ by a given prime.  

\begin{lemma}\label{L:nonunit}
    Let $K$ be an imaginary quadratic field with class number $1$, $\alpha_0\in \mathcal{O}_K$, and let $\alpha$ be an associate of $\alpha_0$. Suppose $\ell$ is a rational prime and $\lambda\in \mathcal{O}_K$ is a prime above $\ell$. If $\alpha_0 \not\equiv u \pmod \lambda$ and $\alpha_0 \not\equiv u \pmod{\bar\lambda}$ for every $u\in \mathcal{O}_K^\times$, then $\ell\nmid N((\alpha-1)(\alpha+1))$.
\end{lemma}
\begin{proof}
    Assume that $\ell \mid N((\alpha-1)(\alpha+1))$. Then, by multiplicativity of the norm, either $\ell\mid (\alpha-1)(\bar\alpha-1)$ or  $\ell\mid (\alpha+1)(\bar\alpha+1)$. Since $\lambda$ is a prime above $\ell$ and $\mathcal{O}_K$ is a unique factorization domain, we have that either $\alpha \equiv \pm 1 \pmod \lambda$ or  $\alpha \equiv \pm 1 \pmod {\bar\lambda}$. Hence $\alpha_0 \equiv u \pmod \lambda$ or $\alpha_0 \equiv u \pmod {\bar\lambda}$ for some $u\in \mathcal{O}_K^\times.$
\end{proof}

For an integer $n\ge 2$, a number field $L$ containing a primitive $n$-th root of unity $\zeta_n$, a prime ideal $\mathfrak{p}\subset \mathcal{O}_L$ coprime to $n$, and $\alpha\in \mathcal{O}_L$, we define the $n$-th power residue symbol $\left(\frac{\alpha}{\mathfrak{p}}\right)_n$ as 

\[\left(\frac{\alpha}{\mathfrak{p}}\right)_n=\begin{cases}
    0 & \text{ if } \alpha \in \mathfrak{p},\\
    \zeta_n^s & \text{ if } \alpha \notin \mathfrak{p}, 
\end{cases} \]
where $\zeta_n^s$ is the unique $n$-th root of unity such that $\alpha^{\frac{N \mathfrak{p}-1}{n}} \equiv \zeta_n^s  \pmod {\mathfrak{p}}$. It is easily seen from the definition that the symbol $ \left(\frac{\cdot}{\mathfrak{p}}\right)_n$ is multiplicative and, for any positive divisor $d$ of $n$,
\begin{equation}\label{E:reduce}
    \left(\frac{\alpha}{\mathfrak{p}}\right)_n^d= \left(\frac{\alpha}{\mathfrak{p}}\right)_{n/d}.
\end{equation}
If $L$ has class number $1$, then $\mathfrak{p}=(\pi)$ for some prime $\pi\in \mathcal{O}_L$, and, as an abuse of notation, we will replace $\mathfrak{p}$ with $\pi$ in the above definition. Let $\omega=e^{2\pi i/3}=(-1+\sqrt{-3})/2$. Following a notion in \cite[\S 9.3]{IR1990}, we say that a prime $\pi \in \ZZ[\omega]$ is \emph{primary} if $\pi \equiv 2 \pmod 3$. The following classical results on reciprocity laws can be found in \cite[\S 9.3]{IR1990} or \cite[\S 7]{Lemmermeyer2000}.

\begin{proposition}[Cubic reciprocity law]\label{P:cubrec}
    Let $\pi_1,\pi_2\in \ZZ[\omega]$ be primary primes such that $N(\pi_1)\ne 3, N(\pi_2)\ne 3,$ and $N(\pi_1)\ne N(\pi_2)$. Then we have 
    \begin{equation*}\
    \left(\frac{\pi_1}{\pi_2}\right)_3=\left(\frac{\pi_2}{\pi_1}\right)_3.
    \end{equation*}
\end{proposition}
Eisenstein established an analogue of Proposition~\ref{P:cubrec} for the sextic residue symbol, commonly known as the sextic reciprocity law. To state the law, we need a notion of \emph{E-primary} primes, which is slightly different from that of primary primes. The original definition of an \textit{E}-primary element in $\ZZ[\omega]$ is due to Eisenstein and is stated in \cite[\S 7.3]{Lemmermeyer2000}. In fact, if $\alpha \in \ZZ[\omega]$ is coprime to $6$, then $\alpha$ is \textit{E}-primary if and only if
\begin{equation*}
    \alpha\equiv \pm 1 \pmod 3 \text{ and } \alpha^3=A+B\omega \text{ with } A+B \equiv 1 \pmod 4
\end{equation*}
(see, e.g., \cite[\S 7]{SS2011}). Using the above criterion, one sees easily that if $\pi\equiv \pm 1 \pmod 3$ is coprime to $6$, then either $\pi$ or $-\pi$ is \textit{E}-primary.

\begin{proposition}[Sextic reciprocity law] \label{P:sexrec}
If $\pi_1,\pi_2\in \ZZ[\omega]$ are \textit{E}-primary primes satisfying $(\pi_1\pi_2,6)=(\pi_1,\pi_2)=1$, then 
\begin{equation*}
    \left(\frac{\pi_1}{\pi_2}\right)_6= (-1)^{\frac{N\pi_1-1}{2}\frac{N\pi_2-1}{2}} \left(\frac{\pi_2}{\pi_1}\right)_6.
\end{equation*}
\end{proposition}

\section{Proof of the main theorem}\label{S:proof}
Throughout this paper, for any elliptic curve $E/\QQ$ which has good reduction at a prime $p$ with the  Frobenius trace $a_p$, we denote
\begin{equation*}\label{E:AEp}
    A_p(E)=(p+1)^2-a_p^2.
\end{equation*}
As pointed out in Odabaş's thesis \cite[Ch.~5]{Odabas2023}, the forward direction of Theorem \ref{T:main} (without using complex multiplication) is obvious and can be elaborated as follows. Suppose $E$ has a $k$-torsion point, say $\alpha$, with rational $x$-coordinate. Then, for any sufficiently large prime $p$, its reduction $\bar{\alpha}$ modulo $p$ must be a root of the (reduced) $k$-th division polynomial, implying that the induced $k$-th Latt\`{e}s map $L_k:\PP^1(\FF_p)\rightarrow \PP^1(\FF_p)$ sends both $\bar{\alpha}$ and $\infty$ to $\infty$. Therefore, it remains to prove the converse. Let $k\in \NN$ and assume that $E$ has no $k$-torsion point with rational $x$-coordinate. Note that $L_1$ is the identity map, which is clearly arithmetically exceptional. Hence we may assume that $k>1$. By Proposition~\ref{P:perm}, it suffices to show that there are infinitely many primes $p$ for which $(A_p(E),k)=1$. We consider $E$ in different cases, depending on its CM discriminant (i.e. the discriminant of $\mathcal{O}$), which is denoted by $D$. All cases rely crucially on the Chinese remainder theorem together with Dirichlet's theorem on arithmetic progressions or its generalization, which can be stated explicitly as follows.
\begin{proposition}\label{P:GDT}
    Let $K$ be an imaginary quadratic field of class number $1$ and let $\mathfrak{m}=(\lambda)$ be a non-zero ideal of $\mathcal{O}_K.$ Then for any $\alpha \in (\mathcal{O}_K/\mathfrak{m})^\times$, there exist infinitely many primes $\pi\in \mathcal{O}_K$ satisfying $\pi \equiv \alpha \pmod{\lambda}.$ More generally, given finitely many pairwise coprime non-zero ideals $\mathfrak{m}_1=(\lambda_1),\ldots,\mathfrak{m}_s=(\lambda_s)$ of $\mathcal{O}_K$ together with residues $\alpha_j\in (\mathcal{O}_K/\mathfrak{m}_j)^\times$, there exist infinitely many primes $\pi\in \mathcal{O}_K$ such that $\pi\equiv \alpha_j \pmod{\lambda_j}$ for all $1\le j\le s$.
\end{proposition}
\begin{proof}
    Recall that the ray class group $\Cl_\mathfrak{m}(K)$ satisfies the exact sequence 
\[U\rightarrow(\mathcal{O}_K/\mathfrak{m})^\times \rightarrow \Cl_\mathfrak{m}(K)\rightarrow \Cl(K)\rightarrow 1,\]
where $U$ is the group of units in $K$ and $\Cl(K)$ is the class group of $K$, which is trivial in this case (see, e.g., \cite[\S V Thm.~1.7]{Milne2020}). Hence we have $\Cl_\mathfrak{m}(K) \simeq (\mathcal{O}_K/\mathfrak{m})^\times/U.$ Now by a generalization of Dirichlet's theorem on arithmetic progressions, as formulated and proven in \cite[Ch.~3--Ch.~6]{Childress2009}, there are infinitely many prime ideals $\mathfrak{p}$ of $\mathcal{O}_K$ in each ray class of $\Cl_\mathfrak{m}(K)$. Since every ideal in $\mathcal{O}_K$ is principal, we immediately obtain the first assertion. For pairwise coprime non-zero ideals $\mathfrak{m}_1,\ldots,\mathfrak{m}_s$, let $\alpha_j\in (\mathcal{O}_K/\mathfrak{m}_j)^\times$ and $\mathfrak{m}=\prod_{j=1}^s\mathfrak{m}_j$. By the Chinese remainder theorem, the residues $\alpha_1,\ldots,\alpha_s$ determine a unique class in $(\mathcal{O}_K/\mathfrak{m})^\times$, so the second assertion follows by applying the first to the modulus $\mathfrak{m}$.
\end{proof}

We shall also frequently refer to a result of Olson \cite{Olson1974} on the classification of torsion subgroups $\Et$ of $E(\QQ)$, which is summarized in first two columns of Table~\ref{tab:summary} below. 
\subsection{Case $D\in \{-4,-8,-16\}$}
First, assume that $D\in \{-4,-16\}$, which is corresponding to the case $K=\QQ(\sqrt{-1})$. Then $\Et$ is isomorphic to $C_2,C_2\times C_2$, or $C_4$. Hence $E$ always has a rational $2$-torsion point and $k$ must be odd. By the Chinese remainder theorem and Dirichlet's theorem, there exist infinitely many primes $p$ such that $p\equiv 3 \pmod 4$ and $p \equiv 1 \pmod k$. Moreover, for all such $p$ we have that $\left(\frac{D}{p}\right)=-1$ (i.e. $p$ is inert in $K$) and, by Proposition~\ref{P:trace},
\[(A_p(E),k)=((p+1)^2,k)=(4,k)=1.\]
On the other hand, if $D=-8$, then $\Et\simeq C_2$, so $k$ is odd. Hence we can apply similar arguments by taking primes $p$ satisfying $p\equiv 5 \pmod 8$ and $p\equiv 1 \pmod k$.
\subsection{Case $D\in \{-7,-28\}$} In this case, we have $\Et\simeq C_2$, so $k$ is odd. If $p$ is a prime satisfying $p\equiv 5 \pmod 7$ and $p\equiv -2 \pmod k$, then $\left(\frac{D}{p}\right)=-1$ and
\[(A_p(E),k)=((p+1)^2,k)=(1,k)=1.\]
Again, the infinitude of such primes $p$ is guaranteed by the Chinese remainder theorem and Dirichlet's theorem.
\subsection{Case $D\in \{-19,-43,-67,-163\}$} Note first that for each $D$ in this set $\left(\frac{D}{2}\right)=\left(\frac{D}{3}\right)=-1$, so $2$ and $3$ are inert primes in $K=\QQ(\sqrt{D}).$ It follows that, for every prime $\lambda\in \mathcal{O}_K$, $\left|\left(\mathcal{O}_K/\lambda\mathcal{O}_K\right)^\times\right|=N(\lambda)-1\ge 3$. Now let the radical of $k$ be written as  $\rad(k)=\ell_1\ell_2\cdots \ell_r$, where $\ell_i's$ are (distinct) rational primes and, for each $1\le i\le r$, let $\lambda_i\in \mathcal{O}_K$ be a prime above $\ell_i$. Also, suppose $(|D|)=(\lambda_D)^2,$ where $\lambda_D\in \mathcal{O}_K$. Using quadratic reciprocity, one sees that, for any prime $p\ne |D|$, $\left(\frac{D}{p}\right)=\left(\frac{p}{|D|}\right)$, so $p$ is inert in $K$ if and only if $p$ is a quadratic non-residue modulo $|D|$. Since there are exactly $\frac{|D|-1}{2}$ quadratic non-residues modulo $|D|$ and $\frac{|D|-1}{2}+2< |D|-1 =\left|\left(\mathcal{O}_K/\lambda_D\mathcal{O}_K\right)^\times\right|$, we have by Proposition~\ref{P:GDT} that there are infinitely many primes $\pi\in \mathcal{O}_K$ such that $\pi \not\equiv \pm 1 \pmod{\lambda_i}$, $\pi \not\equiv \pm 1 \pmod{\bar{\lambda}_i}$, and $\pi \not\equiv t \pmod{\lambda_D}$ for every $1\le i\le r$ and every quadratic non-residue $t$ modulo $|D|$. Hence any prime $\pi$ satisfying the above system of congruences must be a splitting prime. If $N(\pi)=p$, then, using the fact that $\mathcal{O}_K^\times=\{\pm 1\}$, we immediately have from Lemma~\ref{L:nonunit} that $\ell_i \nmid A_p(E)$ for every $i$. Consequently, we can conclude that $(A_p(E),k)=1$ for infinitely many primes $p$.
\subsection{Case $D=-12$} In this case, we have $K=\QQ(\sqrt{-3})$, $\mathcal{O}=\ZZ+2\mathcal{O}_K$, and $j(E)=54000$. Suppose $E$ has a short Weierstrass form $y^2=x^3+Ax+B$, where $A,B\in \ZZ$. Since $j(E)=1728\frac{4A^3}{4A^3+27B^2}$, it can be seen easily that $A=-15u^2$ and $B=22u^3$ for some $u\in \ZZ.$ By a simple calculation, we have that the associated $3$-rd division polynomial can be factored as
\begin{equation*}
     \psi_3(x) = 3x^4+6Ax^2+12Bx-A^2=3(x-3u)(x^3+3ux^2-21u^2x+25u^3).
\end{equation*}
Therefore, $E$ has a $3$-torsion point $P$ such that $x(P)=3u$, which is rational. Moreover, since $\Et\simeq C_2$ or $C_6$, $E$ always has a rational $2$-torsion point, implying $(k,6)=1$. Hence, for any prime $p$ satisfying $p\equiv 2 \pmod 3$ and $p\equiv 1 \pmod k$, \[(A_p(E),k)=((p+1)^2,k)=(4,k)=1.\] Again, since the above congruence system is compatible, there must be an infinite number of primes $p$ which are solutions to the system.
\subsection{Case $D=-27$} In this case, we have $K=\QQ(\sqrt{-3}), \mathcal{O}=\ZZ+3\mathcal{O}_K$, and $j(E)=-2^{15}\cdot 3\cdot 5^3$. By analyzing the $j$-invariant of $E$, one sees that $E$ has a parametrized Weierstrass form $y^2=x^3-120u^2x+506u^3$, where $u\in \ZZ\backslash\{0\}.$ Furthermore, the associated $3$-rd division polynomial factors as
\[\psi_3(x)=3(x-6u)(x^3+6ux^2-204u^2x+800u^3),\]
whence $E$ has a $3$-torsion point with rational $x$-coordinate. It follows that $(k,6)=1$ or $(k,6)=2$. If $(k,6)=1$, as in the previous case, we can simply choose primes $p$ satisfying $p\equiv 2 \pmod 3$ and $p\equiv 1 \pmod k$. Now suppose $(k,6)=2.$ Then $\rad(k)=2m$, where $m=\prod_{i=1}^r \ell_i$ is a (possibly empty) product of distinct primes greater than $3$. For any prime $\ell\notin\{2,3,7\}$ and a prime $\lambda\in \mathcal{O}_K$ above $\ell$, we have $\left|\left(\mathcal{O}_K/\lambda\mathcal{O}_K\right)^\times\right|\ge 12$. Therefore, by Proposition~\ref{P:GDT}, we can find infinitely many primes $\pi$ which simultaneously satisfy the following congruences:
\begin{alignat*}{2}
    \pi &\equiv \omega & &\pmod 2\\
    \pi &\equiv 1 & &\pmod 3 \\
    \pi &\equiv \omega & &\pmod 7 \\
    \pi &\equiv a_i & &\pmod{\lambda_i} \\
    \pi &\equiv b_i & &\pmod{\overline{\lambda_i}},
\end{alignat*}
where $\lambda_i\in \mathcal{O}_K$ is a prime above $\ell_i\ne 7$, $a_i\not\equiv \pm1, \pm\omega,\pm \omega^2 \pmod {\lambda_i}$, and $b_i\not\equiv \pm1, \pm\omega,\pm \omega^2 \pmod {\overline{\lambda_i}}$. (If $\ell_i$ is inert, then we simply consider a single congruence modulo $\ell_i$, instead of those modulo $\lambda_i$ and $\overline{\lambda_i}$.) Since $\pi \equiv 1 \pmod 3$, it follows immediately that $\pi$ is a splitting prime in $\mathcal{O}$. Suppose $N(\pi)=p$. Then by \cite[Theorem~5.3]{RS2009} we have \begin{equation*}
    |E(\FF_p)|=|E(\mathcal{O}_K/\pi\mathcal{O}_K)|=p+1- \left(\frac{6\sqrt{-3}u}{\pi}\right)_2\epsilon(\pi)(\pi+\bar{\pi}),
\end{equation*}
where $\epsilon(\pi)=\pm 1$ (see \cite[Proposition~6.2]{RS2009} for its precise value). Consequently, we find that $|E(\FF_p)|\equiv \pi +\bar{\pi} \equiv \omega+\omega^2\equiv 1 \pmod 2$ and $|E(\FF_p)|\equiv 2\pm (\pi +\bar{\pi}) \equiv 2\pm (\omega+\omega^2)\equiv 2\pm (-1)\not\equiv 0 \pmod 7$. Similarly, if $E'$ is a quadratic twist of $E$, then $|E'(\FF_p)|\not\equiv 0 \pmod 2$ and $|E'(\FF_p)|\not\equiv 0 \pmod 7$. Moreover, by Lemma~\ref{L:nonunit}, $\ell_i\nmid |E(\FF_p)||E'(\FF_p)|$ for every $1\le i\le r$, implying $(A_p(E),k)=(|E(\FF_p)||E'(\FF_p)|,k)=1.$
\subsection{Case $D=-3$} In this case, we have $K=\QQ(\sqrt{-3}), \mathcal{O}=\mathcal{O}_K, j(E)=0$, and $E$ can be defined by the Weierstrass equation $E_d:y^2=x^3+d$, where $d$ is a sixth-power-free integer. Moreover, the torsion structure of $E_d(\QQ)$ can be described explicitly as follows:
\begin{equation}\label{E:torEd}
    E_d^\text{tor}(\QQ)\simeq \begin{cases}
        C_6, & \text{ if } d=1,\\
        C_3, & \text{ if } d\ne 1 \text{ is a square or } d=-432,\\
        C_2, & \text{ if } d\ne 1 \text{ is a cube},\\
        C_1, & \text{ otherwise}.
    \end{cases}
\end{equation}
This classification can be found in Exercise~10.19 of \cite{Silverman2009}. As in the previous case, since $(0,\pm \sqrt{d})$ are $3$-torsion points on $E_d$, either $(k,6)=1$ or $(k,6)=2$. If $(k,6)=1$, then we again choose primes $p$ satisfying $p\equiv 2 \pmod 3$ and $p\equiv 1 \pmod k$. Next, assume that $(k,6)=2$; i.e., $\rad(k)=2\prod_{i=1}^r \ell_i$, where $\ell_i's$ are primes larger than $3$. Our construction for an infinite sequence of primes $p$ with the desired property in this case is more subtle than that in the case $D=-27$, since we use a formula for $|E(\FF_p)|$ that involves the sextic residue symbol, instead of the quadratic residue symbol. More precisely, by Theorem 4 in \cite[Ch.~18]{IR1990}, for any primary splitting prime $\pi\in \mathcal{O}_K$, if $p=N(\pi)$, then 
\begin{equation}\label{E:EdFp}
    |E_d(\FF_p)|=p+1+\overline{\left(\frac{4d}{\pi}\right)}_6 \pi+\left(\frac{4d}{\pi}\right)_6 \bar{\pi} .
\end{equation} 
To make our proof concise, we need the following lemma.
\begin{lemma}\label{L:ab}
    Let $E_d$ and $K$ be defined as above and let $\ell>3$ be a prime such that $\ell\ne 7$. Then there exist $\alpha,\beta \in \mathcal{O}_K$ such that $(\alpha,\ell)=(\beta,\ell)=1$ and, for any prime $\pi\in \mathcal{O}_K$ which is coprime to $6\ell$ and is congruent to $1$ modulo $3$,
    \begin{equation*}
        \left(\frac{\ell}{\pi}\right)_6 \in \begin{cases}
            \{\pm 1\}  &\text{ if } \pi \equiv \alpha\pmod \ell,\\
            \{\pm \omega, \pm \omega^2\} &\text{ if } \pi \equiv \beta \pmod \ell.
        \end{cases}
    \end{equation*}
Moreover, if $N(\pi)=p$, then $\ell\nmid |E_d(\FF_p)||E_d'(\FF_p)|.$
\end{lemma}
\begin{proof}
     We first consider the case when $\ell$ is inert in $K$. Since
$\chi_\ell=(\frac{\cdot}{\ell})_6$ is a multiplicative character of exact order $6$ on
the cyclic group $(\mathcal O_K/\ell\mathcal O_K)^\times$ of order $\ell^2-1$, each sixth root of unity
occurs exactly $(\ell^2-1)/6$ times as a value of $\chi_\ell$. In particular, exactly $(\ell^2-1)/3$ residue classes
are mapped to $\pm1$. Since $5$ is the smallest inert prime in $K$, there are at least $8$ residue classes in $\left(\mathcal{O}_K/\ell\mathcal{O}_K\right)^\times$ which are mapped to $\pm 1$. Among these classes, we may choose $\alpha$ which is not congruent modulo $\ell$ to any $u\in \mathcal{O}_K^\times$. By Proposition~\ref{P:sexrec}, for any prime $\pi\equiv 1 \mod 3$ which is coprime to $6\ell$, if $\pi \equiv \alpha \pmod \ell$, then
    \[\left(\frac{\ell}{\pi}\right)_6=\pm \left(\frac{\pi}{\ell}\right)_6=\pm \left(\frac{\alpha}{\ell}\right)_6\in \{\pm 1\}.\]
    We also have from Lemma~\ref{L:nonunit} that $\ell\nmid |E_d(\FF_p)||E_d'(\FF_p)|,$ where $p=N(\pi)$.
    Since there are at least $16$ classes which are \emph{not} mapped to $\pm 1$ under $\chi_\ell$, we can apply the same arguments to prove the existence of $\beta$ with the desired property.

    Next, consider a split prime $\ell\ge 31;$ i.e., $\ell=\lambda\bar\lambda$, where $\lambda\in \mathcal{O}_K$ is a prime with norm $\ell$. We may also choose $\lambda$ so that $\lambda\equiv \bar\lambda \equiv 1 \pmod 3$. Since $\left|\left(\mathcal{O}_K/\lambda\mathcal{O}_K\right)^\times\right|= \left|\left(\mathcal{O}_K/\bar\lambda\mathcal{O}_K\right)^\times\right| =\ell-1\ge 30,$ there are at least $10$ residue classes modulo $\lambda$ (resp. modulo $\bar\lambda$) which are mapped to $\pm 1$ under $\chi_\lambda$ (resp. $\chi_{\bar\lambda}$) and at least $20$ classes which are mapped to $\pm \omega$ or $\pm \omega^2$. Hence we can find $\alpha_1$ (resp. $\alpha_2$) which is coprime to $\ell$ and is not in the same class with any $u\in \mathcal{O}_K^\times$ modulo $\lambda$ (resp. modulo $\bar\lambda$) such that $\left(\frac{\pi}{\lambda}\right)_6,\left(\frac{\pi}{\bar\lambda}\right)_6\in\{\pm1\}$ for any prime $\pi$ satisfying $\pi \equiv \alpha_1 \pmod \lambda$ and $\pi \equiv \alpha_2 \pmod{\bar\lambda}$. By solving this congruence system and suitably applying the sextic reciprocity law, we can find $\alpha$ coprime to $\ell$ such that, for any prime $\pi\equiv 1 \mod 3$ which is coprime to $6\ell$, if $\pi \equiv \alpha \pmod \ell$, then
    \[\left(\frac{\ell}{\pi}\right)_6=\left(\frac{\lambda}{\pi}\right)_6\left(\frac{\bar\lambda}{\pi}\right)_6=\pm \left(\frac{\pi}{\lambda}\right)_6\left(\frac{\pi}{\bar\lambda}\right)_6\in \{\pm 1\}.\]
    We can do an analogous construction for $\beta$, with the conditions $\left(\frac{\pi}{\lambda}\right)_6 \not\in \{\pm1\}$ and $\left(\frac{\pi}{\bar\lambda}\right)_6\in\{\pm1\}$ for $\pi \equiv \beta \pmod \ell$, so that $\left(\frac{\ell}{\pi}\right)_6\not\in \{\pm1\}$. In both cases, we can again apply Lemma~\ref{L:nonunit} to deduce that $\ell\nmid |E_d(\FF_p)||E_d'(\FF_p)|.$ 

    Finally, we treat the remaining split primes $\ell=13$ and $\ell=19$ separately. Note that $13=\lambda\bar\lambda$, where $\lambda=-1+3\omega$ and $\bar\lambda=-4-3\omega$, which are both E-primary. By direct calculations, we have that $\left(\frac{5}{\lambda}\right)_6=\left(\frac{5}{\bar\lambda}\right)_6=-1$ and $\left(\frac{5\omega}{\lambda}\right)_6=\left(\frac{5\omega}{\bar\lambda}\right)_6=-\omega^2$. Using the sextic reciprocity law, one sees that, for any prime $\pi \equiv 1 \pmod 3$ such that $(\pi,26)=1$, 
    \begin{equation*}
        \left(\frac{13}{\pi}\right)_6 = \begin{cases}
            \pm 1  &\text{ if } \pi \equiv 5 \pmod{13},\\
            \pm \omega  &\text{ if } \pi \equiv 5\omega \pmod{13}.
        \end{cases}
    \end{equation*}
    Since both $5$ and $5\omega$ are not congruent to any $u\in \mathcal{O}_K^\times$ modulo $\lambda$ and $\bar\lambda$, it follows from Lemma~\ref{L:nonunit} that $13 \nmid |E_d(\FF_p)||E_d'(\FF_p)|.$ Similarly, write $19=\lambda\bar\lambda$, where $\lambda=-2+3\omega$ and $\bar\lambda=5+3\omega$. Since $\left(\frac{5}{\lambda}\right)_6=\omega^2, \left(\frac{5}{\bar\lambda}\right)_6=\omega$, $\left(\frac{1+3\omega}{\lambda}\right)_6=\left(\frac{1+3\omega}{\bar\lambda}\right)_6=-\omega^2$, we can conclude that for any prime $\pi \equiv 1 \pmod 3$ such that $(\pi,38)=1$, 
    \begin{equation*}
        \left(\frac{19}{\pi}\right)_6 = \begin{cases}
            \pm 1  &\text{ if } \pi \equiv 5 \pmod{19},\\
            \pm \omega  &\text{ if } \pi \equiv 1+3\omega \pmod{19}.
        \end{cases}
    \end{equation*}
    Moreover, neither $5$ nor $1+3\omega$ is congruent to a unit modulo $\lambda$ and $\bar\lambda$, so $19 \nmid |E_d(\FF_p)||E_d'(\FF_p)|.$
\end{proof}

 Let us now resume the proof of the main result for $D=-3$. By the assumption, $E$ has no rational $2$-torsion point, so it follows from \eqref{E:torEd} that $d$ is not a perfect cube. Suppose $2^t \| d$. Then we can write $d=2^td_1d_2^2d_3^3$, where $d_1$ and $d_2$ are (possibly empty) products of distinct odd primes such that $(d_1,d_2)=1$ and $d_3$ is an odd integer. We divide the proof into three subcases, depending on $t$ modulo $3$. 

 \subsubsection{Subcase $t\equiv 0 \pmod 3$}\label{SS:0}
By \eqref{E:reduce}, Proposition~\ref{P:cubrec}, and the fact that $\left(\frac{\cdot}{\pi}\right)_2=\pm 1$, for any prime $\pi$ such that $(\pi,6)=1$,
\begin{equation}\label{E:case1}
    \left(\frac{4d}{\pi}\right)_6=\pm \left(\frac{d_1}{\pi}\right)_6 \left(\frac{d_2}{\pi}\right)_3\left(\frac{2}{\pi}\right)_3=\pm\left(\frac{d_1}{\pi}\right)_6 \left(\frac{d_2}{\pi}\right)_3\left(\frac{\pi}{2}\right)_3.
\end{equation}
Substituting \eqref{E:case1} into \eqref{E:EdFp} yields
\begin{equation}\label{E:mod2}
\begin{aligned}
     |E_d'(\FF_p)|\equiv  |E_d(\FF_p)|&\equiv p+1\pm \left(\overline{\left(\frac{d_1}{\pi}\right)_6 \left(\frac{d_2}{\pi}\right)_3} +\left(\frac{d_1}{\pi}\right)_6 \left(\frac{d_2}{\pi}\right)_3\right)p\\
     &\equiv \overline{\left(\frac{d_1}{\pi}\right)_6 \left(\frac{d_2}{\pi}\right)_3} +\left(\frac{d_1}{\pi}\right)_6 \left(\frac{d_2}{\pi}\right)_3 \pmod 2.
\end{aligned}
\end{equation}
This implies that, for any primary splitting prime $\pi$ with $N(\pi)=p$, $|E_d(\FF_p)|$ and  $|E_d'(\FF_p)|$ are odd if and only if $\left(\frac{d_1}{\pi}\right)_6 \left(\frac{d_2}{\pi}\right)_3\ne \pm 1$. Let $S$ be the set of odd prime divisors of $d_1d_2k$ and let $q_0$ be the smallest prime divisor of $d_1d_2$. (Note that $q_0$ exists since $3\mid t$ and $d$ is not a perfect cube.) Observe from \eqref{E:reduce} that $\left(\frac{\alpha}{\pi}\right)_3=1$ if and only if $\left(\frac{\alpha}{\pi}\right)_6=\pm 1$. Our strategy is to appropriately construct a congruence system whose prime solutions $\pi$ simultaneously satisfy $\left(\frac{q_0}{\pi}\right)_6\ne \pm 1$, $\left(\frac{q}{\pi}\right)_6= \pm 1$ for each (if any) prime divisor $q\ne q_0$ of $d_1d_2$, and, if $N(\pi)=p$, then $(|E_d(\FF_p)||E_d'(\FF_p)|,\ell)=1$ for any odd prime divisor $\ell$ of $k$. 

Assume first that $q_0=3$. For any prime $\pi\equiv 4-3\omega\pmod 9$, we have $\pi \equiv 1 \pmod 3$ and $N(\pi)\equiv 1\pmod 9$. Hence $-\pi$ is primary and $-\pi=a+b\omega,$ for some $a,b\in\ZZ$ such that $a\equiv 2 \pmod 3$ and $b\equiv 0\pmod 3$, so $b\equiv 0,3,6 \pmod 9$. Since $a=-\pi-b\omega\not\equiv -1 \pmod 9$, we can deduce using the definition of the cubic residue symbol and Theorem~1 in \cite[Ch.~9]{IR1990} that 
\begin{equation*}
\left(\frac{3}{\pi}\right)_6=\left(\frac{-\omega^2(1-\omega)^2}{\pi}\right)_6=\pm\left(\frac{\omega}{\pi}\right)_3\left(\frac{1-\omega}{\pi}\right)_3=\pm\left(\frac{1-\omega}{-\pi}\right)_3\ne \pm 1.
\end{equation*}
By the calculation above, we come up with the following system: $x \equiv 1 \pmod 2$, $x \equiv 4-3\omega \pmod 9$, $x \equiv 1 \pmod 7$, and for each prime $\ell\in S\backslash\{2,3,7\}$, $x \equiv \alpha \pmod \ell$, where $\alpha\in \mathcal{O}_K$ is chosen in accordance with Lemma~\ref{L:ab}. Let $\pi$ be a prime solution to this system with norm $p$. Since $7=\lambda\bar\lambda$, where $\lambda=-1-3\omega$ and $\bar\lambda =2+3\omega$, we have 
\begin{equation*} 
\left(\frac{7}{\pi}\right)_6=\left(\frac{\lambda}{\pi}\right)_6\left(\frac{\bar\lambda}{\pi}\right)_6=\pm \left(\frac{\pi}{\lambda}\right)_6\left(\frac{\pi}{\bar\lambda}\right)_6=\pm 1.
\end{equation*}
Hence, by Lemma~\ref{L:ab} and \eqref{E:mod2}, 
$(|E_d(\FF_p)||E_d'(\FF_p)|,2)=(|E_d(\FF_p)||E_d'(\FF_p)|,\ell_i)=1,$ for $\ell_i\ne 7$. In addition, since $-\pi$ is a primary splitting prime, it follows from \eqref{E:EdFp} and \eqref{E:case1} that
\begin{equation*}
    |E_d(\FF_p)|=p+1\pm \left(\overline{\left(\frac{d_1}{\pi}\right)_6 \left(\frac{d_2}{\pi}\right)_3} +\left(\frac{d_1}{\pi}\right)_6 \left(\frac{d_2}{\pi}\right)_3\right)p\equiv 2\pm (\omega+\omega^2)\not\equiv 0 \pmod 7.
\end{equation*}
Similarly, $|E_d'(\FF_p)|\not\equiv 0 \pmod 7$, whence $(|E_d(\FF_p)||E_d'(\FF_p)|,7)=1$. Therefore, we can conclude that $(|E_d(\FF_p)||E_d'(\FF_p)|,k)=1.$ The infinitude of such primes $p$ follows immediately from Proposition~\ref{P:GDT} applied to the proposed congruence system.

Suppose $q_0=5$. As proven in Lemma~\ref{L:ab}, we can choose, for example, $\beta=2\omega$, so that $\left(\frac{5}{\pi}\right)_6\ne \pm 1$ and $5 \nmid |E_d(\FF_p)||E_d'(\FF_p)|$ for any prime $\pi$ satisfying the assumption in the lemma. Hence we can simply choose $\pi$ to be any prime such that $\pi \equiv 1 \pmod 2, \pi \equiv 1 \pmod 3, \pi \equiv 2\omega \pmod 5, \pi \equiv 1 \pmod 7$, and $\pi \equiv \alpha \pmod \ell$ for each $\ell\in S\backslash\{2,3,5,7\}$. Since this system is compatible, there must be infinitely many choices for $\pi$.

Next, assume that $q_0=7.$ Consider the congruence system
\begin{alignat*}{2}
    x &\equiv 1 & &\pmod 2\\
    x &\equiv 1 & &\pmod 3 \\
    x &\equiv  \begin{cases}
    					\omega & \text{if } 7\mid d_1\\
    					1-\omega & \text{if } 7\mid d_2
    				\end{cases} 
    				  & &\pmod 7\\
    x &\equiv \alpha & &\pmod \ell, 
\end{alignat*}
for each $\ell\in S\backslash\{2,3,7\}$ with $\alpha$ (depending on $\ell$) defined in Lemma~\ref{L:ab}. By direct calculations, for any prime $\pi\equiv \omega \pmod 7$, we have $\left(\frac{7}{\pi}\right)_6=\pm\omega^2$, $\left(\frac{7}{\pi}\right)_3=\pm\omega$, and $N(\pi)\equiv 1 \pmod 7$. 
Similarly, if $\pi \equiv 1-\omega \pmod 7$, then $\left(\frac{7}{\pi}\right)_6=\pm\omega$, $\left(\frac{7}{\pi}\right)_3=\pm\omega^2$, and $N(\pi)\equiv 3 \pmod 7$. Therefore, we can deduce using \eqref{E:EdFp}, \eqref{E:case1}, and Lemma~\ref{L:ab} that, for any prime $\pi$ satisfying the above system, if $N(\pi)=p$, then $(|E_d(\FF_p)||E_d'(\FF_p)|,2)=(|E_d(\FF_p)||E_d'(\FF_p)|,\ell_i)=1$ for $1\le i\le r$, implying $(|E_d(\FF_p)||E_d'(\FF_p)|,k)=1.$

For the remaining cases $q_0>7$, we again apply Lemma~\ref{L:ab} and solve the congruence system $x \equiv 1 \pmod 2, x \equiv 1 \pmod 3, x \equiv 1 \pmod 7, x \equiv \beta \pmod{q_0},$ and $x \equiv \alpha \pmod \ell$ for each $\ell\in S\backslash\{2,3,7,q_0\}$, in order to obtain infinitely many primes $\pi$ with $N(\pi)=p$ satisfying $(|E_d(\FF_p)||E_d'(\FF_p)|,k)=1.$ 

\begin{remark}
If $d$ is a perfect cube (i.e. $t\equiv 0\pmod 3$ and $d_1=d_2=1$), then we have from \eqref{E:mod2} that $|E_d(\FF_p)|\equiv 0 \pmod 2$ for any odd prime $p$ of good reduction. Thus if $(k,6)=2$, then the $k$-th Latt\`es map attached to $E_d$ is \emph{not} arithmetically exceptional. Hence the assumption that $d$ is not a perfect cube is crucial in the proof for this subcase.
\end{remark}

\subsubsection{Subcase $t\equiv 1 \pmod 3$} By the same calculation as in \eqref{E:case1}, we have that
\begin{equation*}
     \left(\frac{4d}{\pi}\right)_6 =\pm \left(\frac{d_1}{\pi}\right)_6 \left(\frac{d_2}{\pi}\right)_3 .
\end{equation*}
Hence for any primary splitting prime $\pi$ with norm $p$ 
\begin{equation*}
\begin{aligned}
     |E_d'(\FF_p)|\equiv  |E_d(\FF_p)|&=p+1\pm \left(\overline{\left(\frac{d_1}{\pi}\right)_6 \left(\frac{d_2}{\pi}\right)_3}\pi +\left(\frac{d_1}{\pi}\right)_6 \left(\frac{d_2}{\pi}\right)_3\bar\pi\right)\\
     &\equiv \overline{\left(\frac{d_1}{\pi}\right)_6 \left(\frac{d_2}{\pi}\right)_3}\pi +\left(\frac{d_1}{\pi}\right)_6 \left(\frac{d_2}{\pi}\right)_3\bar\pi \pmod 2.
\end{aligned}
\end{equation*}
If $d_1d_2=1$, then $|E_d'(\FF_p)|\equiv  |E_d(\FF_p)| =p+1\pm(\pi+\bar\pi)$. Choosing primes $\pi$ with norm $p$ such that $\pi \equiv \omega \pmod 2, \pi \equiv 1 \pmod 3, \pi \equiv \omega \pmod 7,$ and $\pi\equiv \alpha \pmod \ell$ for each odd prime divisor $\ell\ne 7$ of $k$, we obtain infinitely many primes $p$ for which $(|E_d(\FF_p)||E_d'(\FF_p)|,k)=1.$ On the other hand, if $d_1d_2>1$, then we can repeat the arguments in the previous subcase to obtain a compatible congruence system which generates infinitely many primes $p$ satisfying $(|E_d(\FF_p)||E_d'(\FF_p)|,k)=1.$

\subsubsection{Subcase $t\equiv 2 \pmod 3$} In this subcase, we have 
$$
\begin{aligned}
\left|E_d^{\prime}\left(\mathbb{F}_p\right)\right| \equiv\left|E_d\left(\mathbb{F}_p\right)\right|= & p+1 \pm\left(\overline{\left(\frac{d_1}{\pi}\right)_6\left(\frac{d_2}{\pi}\right)_3\left(\frac{4}{\pi}\right)_3} \pi+\left(\frac{d_1}{\pi}\right)_6\left(\frac{d_2}{\pi}\right)_3\left(\frac{4}{\pi}\right)_3 \bar{\pi}\right) \\
& \equiv\overline{\left(\frac{d_1}{\pi}\right)_6\left(\frac{d_2}{\pi}\right)_3 \pi}+\left(\frac{d_1}{\pi}\right)_6\left(\frac{d_2}{\pi}\right)_3 \pi \quad \pmod 2
\end{aligned}
$$
for any primary splitting prime $\pi$ such that $N(\pi)=p.$ 
We can then imitate the arguments in the subcase $t\equiv 1 \pmod 3$ to obtain the desired result.

%Our goal is to show, with the aid of Lemma~\ref{L:ab}, that there exist a compatible system of congruences whose solutions $\pi$ satisfy the latter condition. Moreover, for each solution $\pi$ of this system, if $N(\pi)=p$, then we have $(|E_d(\FF_p)||E_d'(\FF_p)|,l)=1$ for any odd prime divisor $l$ of $k$.

%Then the moduli involved in the desired congruence system are precisely those in the set $S$. Using reciprocity laws and the fact that the values of $\left(\frac{\pi}{\lambda}\right)_n$ distribute equally among the $n^\text{th}$-roots of unity, we will show that it is always possible to construct such a compatible system. 

%%%%%%%%%%%%%%%%%%%%%%%%%%%%%%%%%%%%%%%%%%%
% REVISED SECTION 4 (clean; merged)
\section{The case $D = -11$}\label{S:11}

In this section we treat the remaining CM field $K=\mathbb{Q}(\sqrt{-11})$, which is excluded from
the result proved in Section~\ref{S:proof}. The obstruction comes from the fact that $3$ splits in $K$ and,
for split primes $p\neq 3$ of good reduction for $E$, forces a factor of $3$ in $A_p(E)$.
Combined with the unavoidable factor $2$ coming from inert primes $p>2$, this yields a genuine obstruction precisely when $6\mid k$.

Let $K=\mathbb{Q}(\sqrt{-11})$ and let $\mathcal{O}=\mathcal{O}_K=\mathbb{Z}[\theta]$, where
$\theta=(1+\sqrt{-11})/2$. Since $N(\theta)=3$, we have that $3$ splits
in $K$. Note also that $\mathcal{O}^\times=\{\pm1\}$. 
%Since $h_K=1$, every nonzero ideal of $\mathcal{O}$ is principal. We will write $\lambda\subset \mathcal{O}$ for prime ideals (especially in congruence conditions), and $\bar{\lambda}$ for the conjugate ideal. When we say that $\pi\in \mathcal{O}$ is a \emph{prime element}, we mean that $(\pi)$ is a nonzero prime ideal of $\mathcal{O}$. Note that if a rational prime $\ell$ is inert or ramified in $K$, then the unique prime ideal $\lambda\mid \ell$ is fixed by conjugation, i.e.\ $\bar{\lambda}=\lambda$. 
For the remainder of this section, we assume that $E/\mathbb{Q}$
has CM by $\mathcal{O}$.

\begin{lemma}\label{L:D11_forced3}
Let $p$ be a prime of good reduction for $E$ with $p\ne 3$ and $\left(\frac{-11}{p}\right)=1$.
Then $3 \mid A_p(E)$.
\end{lemma}

\begin{proof}
Since $\left(\frac{-11}{p}\right)=1$, Proposition~\ref{P:trace} gives
$a_p=\pi+\bar{\pi}$ for some prime $\pi\in \mathcal{O}$ with $N(\pi)=p$. 
Let $\mathfrak{q}=(\theta)$. Then $\mathfrak{q}$ is a prime ideal above $3$
and $\mathcal{O}/\mathfrak{q}\simeq \FF_3$. Because $p\neq 3$, we have $\pi\notin \mathfrak{q}$. Since $1$ and $-1$ are in different residue classes modulo $\mathfrak{q}$, we have $\pi\equiv \pm 1 \pmod{\mathfrak{q}}$, i.e.\ $\pi-1\in \mathfrak{q}$ or $\pi+1\in \mathfrak{q}$.
As noted after Proposition~\ref{P:trace}, it follows that
\[
3=N(\mathfrak{q})\mid N((\pi-1)(\pi+1))=A_p(E),
\]
as desired.
\end{proof}

\begin{proposition}\label{P:D11_6divk_notexc}
If $6\mid k$, then the $k$-th Latt\`{e}s map $L_k$ attached to $E$ is not arithmetically exceptional.
\end{proposition}

\begin{proof}
Let $p>2$ be a prime of good reduction for $E$ with $p\nmid 33$.

If $\left(\frac{-11}{p}\right)=-1$, then Proposition~\ref{P:trace} gives $a_p=0$, whence
$A_p(E)=(p+1)^2$
is even. Since $k$ is also even, we have $\bigl(A_p(E),k\bigr)\ge 2$.

If $\left(\frac{-11}{p}\right)=1$, then Lemma~\ref{L:D11_forced3} gives $3\mid A_p(E)$, which implies that $\bigl(A_p(E),k\bigr)\ge 3$.

In summary, $\bigl(A_p(E),k\bigr)\neq 1$ for every prime $p>2$ of good reduction with $p\nmid 33$.
By Proposition~\ref{P:perm}, $L_k$ fails to permute $\PP^1(\FF_p)$ for all but finitely many primes $p$,
and hence $L_k$ is not arithmetically exceptional.
\end{proof}

%As explained at the beginning of Section~\ref{S:proof}, the ``only if'' direction of Odaba\c{s}'s conjecture holds in general: if $E$ has a $k$-torsion point whose $x$-coordinate is rational, then $L_k$ cannot induce a permutation of $\PP^1(\FF_p)$ for all sufficiently large primes $p$.
We now prove that the converse of Proposition~\ref{P:D11_6divk_notexc} also holds.

\begin{proposition}\label{P:D11_6notdivk_exc}
If $6\nmid k$, then the $k$-th Latt\`{e}s map $L_k$ attached to $E$ is arithmetically exceptional.
\end{proposition}

\begin{proof}

 As before, by Proposition~\ref{P:perm}, it suffices to show that there are infinitely many primes $p$ such that $(|E(\FF_p)||E'(\FF_p)|,k)=1,$ where $E'$ is a quadratic twist of $E$. Assume first that $3\mid k$. Then $\rad(k)=3\ell_1\ell_2\cdots \ell_r$, where $\ell_i\ge 5$. By Dirichlet's theorem, there exists infinitely many primes $p$ satisfying $p\equiv 1 \pmod 3$, $p \equiv 2 \pmod{11}$, and $p\equiv 1 \pmod{\ell_i}$ for $1\le i \le r$ with $\ell_i\ne 11$. For any such prime $p$, we can use quadratic reciprocity to obtain
\[
\left(\frac{-11}{p}\right)=\left(\frac{p}{11}\right)=\left(\frac{2}{11}\right)=-1.
\]
Hence, by Proposition~\ref{P:trace}, for any good prime $p$ satisfying the above congruence system, $|E(\FF_p)||E'(\FF_p)|=(p+1)^2\equiv 1 \pmod 3$ and $|E(\FF_p)||E'(\FF_p)|\not \equiv 0 \pmod{\ell_i}$ for $1\le i\le r.$ Therefore, we have $(|E(\FF_p)||E'(\FF_p)|,k)=1$ for infinitely many primes $p$.

Next, assume that $3\nmid k$. Then we can write $\rad(k)=\ell_1\ell_2\cdots \ell_r$ for some distinct primes $\ell_i's$, none of which is equal to $3$. For each $1\le i\le r,$ let $\lambda_i\in \mathcal{O}$ be a prime above $\ell_i$. Since $2$ is inert in $K$, we have $|\left(\mathcal{O}/\lambda_i\mathcal{O}\right)^\times|\ge 3$. Hence we can find $a_i,b_i\in \mathcal{O}$ such that $a_i\not\equiv \pm 1 \pmod{\lambda_i}$ and $b_i\not\equiv \pm 1 \pmod{\overline{\lambda_i}}$. By Proposition~\ref{P:GDT}, there are infinitely many primes $\pi\in \mathcal{O}$ which simultaneously satisfy $\pi \equiv 3 \pmod{11}$, $\pi \equiv a_i \pmod{\lambda_i},$ and $\pi \equiv b_i \pmod{\overline{\lambda_i}}$ for all $1\le i\le r$ with $\lambda_i\nmid 11$. (If $\ell_i$ is inert; i.e., $\lambda_i=\overline{\lambda_i}$, we simply take $a_i=b_i$.) The first congruence implies that these are splitting primes and it is clear that $3\not\equiv \pm 1 \pmod{\sqrt{-11}}$. Thus it follows immediately from Lemma~\ref{L:nonunit} that, if $N(\pi)=p$, then $\ell_i\nmid |E(\FF_p)||E'(\FF_p)|$ for every $1\le i \le r$, implying $(|E(\FF_p)||E'(\FF_p)|,k)=1$, as desired.
\end{proof}

By Proposition~\ref{P:D11_6notdivk_exc}, we immediately have the following result.
\begin{corollary}
    Let $E/\QQ$ be an elliptic curve having complex multiplication by $\QQ(\sqrt{-11})$ and let $k\in \NN$ with $6\nmid k$. Then Conjecture~\ref{C:Obadas} is true for the Latt\`es map $L_k$ attached to $E$.
\end{corollary}

The next Proposition provides explicit CM curves showing that the conjecture fails when $6\mid k$. We recall the explicit Weierstrass models for elliptic curves over $\QQ$ with complex
multiplication by an order $\mathcal{O}$ of class number one.
Following Jim\'enez Urroz \cite{JimenezUrrozANTS}, any CM elliptic curve over $\QQ$
with discriminant $D\in\{-3,-4,-7,-8,-11,-19,-43,-67,-163\}$ admits a short Weierstrass equation of the form
\[
E_g:\quad y^2 = x^3 + g_4(g)\,x + g_6(g),
\]
where $g\in\QQ^\times$ and the coefficients $(g_4(g),g_6(g))$ are given explicitly as
quadratic and cubic polynomials in $g$, depending on $D$.
In particular, for $D=-11$, one may take
\[
(g_4(g),g_6(g))=\left(-\frac{2\cdot 11}{3}g^2,\,-\frac{7\cdot 11^2}{108}g^3\right),
\]
which yields a CM model with $j$-invariant $-2^{15}$.
By setting $g=-6u$, this family specializes to the model $E_u$ below.

%%%%%%%%%% generalize to parameter

\begin{proposition}\label{P:D11_counterexample_curve}
Let $u\in\QQ^\times$ and consider the elliptic curve
\[
E_u:\ y^2=x^3-264u^2x+1694u^3.
\]
 Then $E_u$ has complex multiplication by $\mathcal{O}$ and, for every integer $k> 1$, there is no $k$-torsion point $P\in E_u(\Qbar)$ such that $x(P)\in\QQ$.
\end{proposition}

\begin{proof}
A direct computation shows that $j(E_u)=-2^{15}$, so $E_u$ has CM by $\mathcal{O}$. Suppose $P\in E_u(\Qbar)$ is a torsion point with $x(P)\in\QQ$ and let $n>1$ be the order of $P$. For any prime $q$ dividing $n$, the point $T:=[n/q]P$ has order $q$, and its $x$-coordinate is rational, since for each $m\ge 1$ we have
\[
x([m]P)=L_m(x(P)),
\]
where $L_m\in\QQ(x)$ is the $m$-th Latt\`es map attached to $E_u$. Then the rationality of $x(T)$ implies $\sigma(T)=\chi(\sigma)T$ for all $\sigma\in\Gal(\Qbar/\QQ)$, where $\chi$ is a quadratic (possibly trivial) character. Hence $T$ is a rational point of order $q$ on the quadratic twist $E_u^\chi$ of $E_u$. Since a quadratic twist preserves the $j$-invariant, it follows that $E_u^\chi$ has CM by $\mathcal{O}$ and therefore has no non-trivial rational torsion point by Olson's classification \cite{Olson1974}, a contradiction.
\end{proof}

%%%%%%%%%%%%%%%%%%%%%%%%%%%%%
\section{Examples}\label{S:examples}

In this section, we illustrate the main theorem and the counterexamples for $D=-11$ with a small number of representative examples. Table~\ref{tab:summary} records, for each CM discriminant $D$, the possible torsion groups $\Et$, the necessary condition on $k$ for $L_k$ to be arithmetically exceptional (forced by the rational torsion), and the prime-selection strategy used in the proof. The first two subsections illustrate the two strategies appearing in the proof of Theorem~\ref{T:main}, namely congruence conditions on inert primes, and Chebotarev conditions on split primes, and the last subsection illustrates examples and counterexamples for $D=-11$. All data below were verified using \texttt{SageMath}~\cite{sagemath}; analogous data for the remaining discriminants, together with the code used, can be accessed at \cite{SSP26}. 

\begin{table}[h]
\centering
\renewcommand{\arraystretch}{1.2}
\caption{Prime selection strategies by CM discriminant}
\label{tab:summary}
\small
\begin{tabular}{c c l l}
\toprule
$D$ & $\Et$ & Condition on $k$ & Prime selection strategy \\
\midrule
$-4,-16$ 
  & $C_2$, $C_2\!\times\!C_2$, $C_4$
  & $k$ odd
  & $p \equiv 3 \pmod{4}$, $p \equiv 1 \pmod{k}$ \\
\hline
$-8$ 
  & $C_2$
  & $k$ odd
  & $p \equiv 5 \pmod{8}$, $p \equiv 1 \pmod{k}$ \\
\hline
$-7,-28$
  & $C_2$
  & $k$ odd
  & $p \equiv 5 \pmod{7}$, $p \equiv -2 \pmod{k}$ \\
\hline
$-19,-43,-67,-163$
  & $C_1$
  & any $k\ge 2$
  & Chebotarev (splitting primes) \\
\hline
$-12$
  & $C_2$, $C_6$
  & $ (k,6)=1$
  & $p \equiv 2 \pmod{3}$, $p \equiv 1 \pmod{k}$ \\
\hline
\multirow{2}{*}{$-27$}
  & \multirow{2}{*}{$C_1$, $C_3$}
  & (i) $ (k,6)=1$
  & (i) $p \equiv 2 \pmod{3}$, $p \equiv 1 \pmod{k}$\\
  &
  &(ii) $ (k,6)=2$
  &(ii) Chebotarev (splitting primes)\\
\hline
\multirow{2}{*}{$-3$}
  & \multirow{2}{*}{$C_1$, $C_2$, $C_3$, $C_6$}
  & (i) $ (k,6)=1$
  & (i) $p \equiv 2 \pmod{3}$, $p \equiv 1 \pmod{k}$\\
  &
  & (ii) $ (k,6)=2$
  & (ii) Chebotarev (splitting primes)\\
\hline
\multirow{2}{*}{$-11$}
  & \multirow{2}{*}{$C_1$}
  & (i) $\rad(k)=3\prod_{\ell_i>3} \ell_i$
  & (i) $p \equiv 1 \pmod{3,\ell_i}$, $p \equiv 2 \pmod{11}$ \\
  &
  & (ii) $3 \nmid k$
  & (ii) Chebotarev (splitting primes) \\

\bottomrule
\end{tabular}
\normalsize
\end{table}

\subsection{An inert-prime example} 

Consider $E\colon y^2 = x^3 + x$, which has $j(E) = 1728$, CM discriminant $D=-4$, and $\Et \simeq C_2$, generated by the rational $2$-torsion point $(0,0)$. The torsion forces $k$ to be odd. The $3$-rd division polynomial $\psi_3(x) = 3x^4 + 6x^2 - 1$ has no rational roots, confirming that $E$ has no rational $3$-torsion $x$-coordinate; thus the hypothesis of Theorem~\ref{T:main} holds for $k = 3$. The $3$-rd Latt\`es map attached to $E$ is
\[
L_3(x) = \frac{x^9 - 12x^7 + 30x^5 + 36x^3 + 9x}{9x^8 + 36x^6 + 30x^4 - 12x^2 + 1}.
\]
Applying the prime selection strategy in Table~\ref{tab:summary}, we have that $L_3$ permutes $\PP^1(\FF_p)$ for any prime $p\equiv 7 \pmod{12}$. Table~\ref{tab:D4} includes data for both inert and split primes. 

We note in passing that the inert-prime strategy occasionally produces degenerate permutations. For example, for $E\colon y^2 = x^3 - 15x + 22$, which has CM discriminant $D=-12$, the choice $p = 11$ makes $L_5$ the identity map on $\PP^1(\FF_{11})$.

\begin{table}[h]
\centering
\caption{Permutation behavior of $L_3$ on $\PP^1(\FF_p)$ for $E\colon y^2 = x^3 + x$}\label{tab:D4}
\begin{tabular}{|c|c|c|c|c|}
\hline
$p$ & $\left(\frac{-1}{p}\right)$ & $a_p$ & $ (A_p(E),3)$ & $L_3$ permutes $\PP^1(\FF_p)$? \\
\hline
$5$ & $+1$ & $2$ & $1$ & Yes \\
$7$ & $-1$ & $0$ & $1$ & Yes \\
$11$ & $-1$ & $0$ & $3$ & No \\
$13$ & $+1$ & $-6$ & $1$ & Yes \\
$19$ & $-1$ & $0$ & $1$ & Yes \\
$23$ & $-1$ & $0$ & $3$ & No \\
$29$ & $+1$ & $10$ & $1$ & Yes \\
$31$ & $-1$ & $0$ & $1$ & Yes \\
\hline
\end{tabular}
\end{table}

\subsection{A split-prime example}

Consider $E\colon y^2 = x^3 + 2$, which has $j(E) = 0$ and CM discriminant $D=-3$. Since $d=2$ is neither a perfect square nor a perfect cube, the classification~\eqref{E:torEd} gives $\Et$ trivial. The $3$-rd division polynomial, however, has $x = 0$ and $x = -2$ as rational roots, so the condition $ (k,6) \in \{1,2\}$ is necessary for $L_k$ to be arithmetically exceptional. On the other hand, since $E$ has no rational $2$-torsion point, Theorem~\ref{T:main} implies that the $2$-nd Latt\`es map
\[
L_2(x) = \frac{x^4 - 16x}{4(x^3 + 2)}
\]
is indeed arithmetically exceptional.
Table~\ref{tab:D3} shows the permutation behavior of $L_2$ on $\PP^1(\FF_p)$. All inert primes  $p \equiv 2 \pmod{3}$ fail, since $A_p(E) = (p+1)^2$ is even. Among split primes, $L_2$ permutes $\PP^1(\FF_p)$ exactly when $a_p$ is odd, i.e.\ when the corresponding $\pi \in \mathcal{O}_K$ satisfies $\pi \not\equiv  1 \pmod{2}$. Observe that the failures at $p = 31$ and $p=43$ arise because the corresponding \emph{primary} prime $\pi$ satisfies $\pi \equiv 1 \pmod 2$. We may thus choose $p=N(\pi)$, where $\pi \equiv \omega \pmod 2$ and $\pi \equiv 1 \pmod 3$; Proposition~\ref{P:GDT} guarantees infinitely many such primes.

\begin{table}[h]
\centering
\caption{Permutation behavior of $L_2$ on $\PP^1(\FF_p)$ for $E\colon y^2=x^3+2$ }\label{tab:D3}
\begin{tabular}{|c|c|c|c|c|}
\hline
$p$ & $\left(\frac{-3}{p}\right)$ & $a_p$ & $ (A_p(E),2)$ & $L_2$ permutes $\PP^1(\FF_p)$? \\
\hline
$5$ & $-1$ & $0$ & $2$ & No \\
$7$ & $+1$ & $-1$ & $1$ & Yes \\
$11$ & $-1$ & $0$ & $2$ & No \\
$13$ & $+1$ & $-5$ & $1$ & Yes \\
$19$ & $+1$ & $7$ & $1$ & Yes \\
$31$ & $+1$ & $-4$ & $2$ & No \\
$37$ & $+1$ & $-11$ & $1$ & Yes \\
$43$ & $+1$ & $-8$ & $2$ & No \\
\hline
\end{tabular}
\end{table}

\subsection{Examples for $D=-11$}

We now demonstrate that Conjecture~\ref{C:Obadas} fails when $K = \QQ(\sqrt{-11})$ and $6 \mid k$, and that the positive result for $6 \nmid k$ is sharp.

\begin{example}\label{ex:D11}
Let $E\colon y^2 = x^3 - 264x + 1694$, the curve $E_u$ of Proposition~\ref{P:D11_counterexample_curve} with $u=1$. It has $j(E) = -2^{15}$, CM by $\mathcal{O}_K = \ZZ\bigl[\frac{1+\sqrt{-11}}{2}\bigr]$, and trivial rational torsion. The $2$-nd and $3$-rd division polynomials factor as
\begin{align*}
\psi_2(x) &= 4x^3 - 1056x + 6776 \qquad \text{(irreducible over }\QQ\text{)}, \\
\psi_3(x) &= 3\,(x^2 - 22x + 132)\,(x^2 + 22x - 176),
\end{align*}
where both quadratic factors of $\psi_3(x)$ are irreducible. Hence $E$ has no $k$-torsion point with rational $x$-coordinate for any $k$ divisible by $6$ (see Proposition~\ref{P:D11_counterexample_curve}). The $2$-nd Latt\`es map is
\[
L_2(x) = \frac{x^4 + 528x^2 - 13552x + 69696}{4(x^3 - 264x + 1694)},
\]
while $L_3$ has degree $9$ and $L_6 = L_2 \circ L_3 = L_3 \circ L_2$ has degree $36$.

By Proposition~\ref{P:D11_6divk_notexc}, $L_6$ never permutes $\PP^1(\FF_p)$ for any prime $p \ne 11$ of good reduction since $2 \mid  (A_p(E), 6)$ if $p$ is inert and $3 \mid  (A_p(E), 6)$ if $p$ splits. Tables~\ref{tab:vs_D11_5} and~\ref{tab:vs_D11_7} record the values of $L_6$ on $\PP^1(\FF_5)$ and $\PP^1(\FF_7)$. Observe that the distinct collapse patterns reflect the two different sources of obstruction. The six points of $\PP^1(\FF_5)$ collapse to two image values, each with fibre of size $3$. This reflects the non-injectivity of $L_3$ over the split prime $p=5$, while $L_2$ is itself a bijection on $\PP^1(\FF_5)$. Over $\FF_7$, the failure occurs in pairs, namely $L_6(1)=L_6(2)=3$, $L_6(3)=L_6(4)=0$, and $L_6(5)=L_6(6)=4$. This reflects the non-injectivity of $L_2$ over the inert prime $p=7$, while $L_3$ is itself a bijection on $\PP^1(\FF_7)$. Since $L_6 = L_3 \circ L_2 = L_2 \circ L_3$, the non-injectivity propagates from whichever factor fails.

\begin{table}[h]
\centering
\caption{Values of $L_6$ on $\PP^1(\FF_5)$ for $E\colon y^2=x^3-264x+1694$}\label{tab:vs_D11_5}
\begin{tabular}{|c|cccccc|}
\hline
$x$ & $0$ & $1$ & $2$ & $3$ & $4$ & $\infty$ \\
\hline
$L_6(x)$ & $3$ & $3$ & $3$ & $\infty$ & $\infty$ & $\infty$ \\
\hline
\end{tabular}
\end{table}

\begin{table}[h]
\centering
\caption{Values of $L_6$ on $\PP^1(\FF_7)$ for $E\colon y^2=x^3-264x+1694$}\label{tab:vs_D11_7}
\begin{tabular}{|c|cccccccc|}
\hline
$x$ & $0$ & $1$ & $2$ & $3$ & $4$ & $5$ & $6$ & $\infty$ \\
\hline
$L_6(x)$ & $\infty$ & $3$ & $3$ & $0$ & $0$ & $4$ & $4$ & $\infty$ \\
\hline
\end{tabular}
\end{table}

In contrast, Proposition~\ref{P:D11_6notdivk_exc} asserts that $L_k$ \emph{is} arithmetically exceptional whenever $6\nmid k$. For instance, for $k=7$, choosing $p=N(\pi)$ with $\pi \equiv 3 \pmod{11}$ and $\pi \equiv 2 \pmod{7}$ produces infinitely many permutation primes. We have that $L_7$ permutes $\PP^1(\FF_p)$ for $p \in \{5, 7, 17, 23, 31\}$, while failing at the inert prime $p=13$, since $7 \mid p+1$. On $\PP^1(\FF_5)$, for example, $L_7$ acts as the $3$-cycle $(0\,1\,2)$, fixing $3$, $4$, and $\infty$.
\end{example}

\section{An outlook for further exploration}\label{S:outlook}
Theorem~\ref{T:main} verifies Odaba\c{s}'s conjectural criterion for all CM elliptic curves over $\QQ$
except for the unique CM field $\QQ(\sqrt{-11})$, where Proposition~\ref{P:D11_6divk_notexc} shows a genuinely new obstruction when $6\mid k$. The key ingredient in the CM case is the availability of explicit formulas for the Frobenius traces $a_p$ in terms of arithmetic data associated with the CM field. No comparable general description is known for non-CM elliptic curves. Therefore, the methods developed in Section~\ref{S:proof} and Section~\ref{S:11}, which depend essentially on these formulas, Proposition~\ref{P:perm}, and Olson's classification of torsion subgroups of CM elliptic curves do not readily generalize to the non-CM setting. This naturally suggests several directions in which the present work can be extended or refined. We single out, in particular, a uniform treatment of the case $k=2$ (which does not require CM),
and then outline broader problems motivated by a Galois-representation perspective. The proof of the following proposition stems from Elkies' answer in a MathOverflow discussion; see also an alternative approach using Galois representations due to Bellaïche \cite{BE2014}. In this section, we shall denote the group of $n$-torsion points on $E/K$ by $E(K)[n]$ and let $E[n]:=E(\overline{K})[n]$.

\begin{proposition}\label{P:k2_general}
Let $E/\QQ$ be an elliptic curve. Then the second Latt\`es map $L_2$ attached to $E$ is arithmetically exceptional if and only if $E$ has no non-trivial rational $2$-torsion point.
\end{proposition}

\begin{proof}
We first write $E: y^2=P(x)$, where $P(x)\in \ZZ[x]$ is a cubic polynomial. As mentioned in Section~\ref{S:proof}, the forward direction is obvious. To prove the converse, assume that $E$ has no non-trivial rational $2$-torsion point; i.e., $P(x)$ is irrational over $\QQ$. Let $L$ be the splitting field of $P(x)$. Then $\Gal(L/\QQ)$ contains a $3$-cycle, so it must be either $S_3$ or $C_3$. For any prime $p$ which is unramified in $L$ and a prime ideal $\mathfrak{p}$ above $p$, the Frobenius element $\mathrm{Frob}_\mathfrak{p}\in \Gal(L/\QQ)$ acts on the three roots of $P(x)$ and its cycle type determines the factorization of $P(x)$ over $\FF_p$.
In particular, $\mathrm{Frob}_\mathfrak{p}$ is a $3$-cycle if and only if $P(x)$ is irreducible over $\FF_p$. By the Chebotarev density theorem \cite[Thm.~13.4]{Neukirch1999} applied to $L/\QQ$, there are infinitely many primes $p$ of good reduction for which $\mathrm{Frob}_\mathfrak{p}$ is a $3$-cycle, hence for which $P(x)\bmod p$ is irreducible. For each such prime $p$, we have that $E(\FF_p)$ contains only torsion points of odd order, so $|E(\FF_p)|\equiv 1 \pmod 2$. Since $a_p=p+1-|E(\FF_p)|$ for every good prime $p$, we have that $a_p\equiv 1 \pmod 2$ for infinitely many primes $p$. Hence, by Proposition~\ref{P:perm}, it follows immediately that $L_2$ is arithmetically exceptional.
\end{proof}

\begin{remark}
The following remarks discuss some consequences and related observations stemming from Proposition~\ref{P:k2_general}.
\begin{itemize}
\item [(i)] By slightly modifying the arguments in the proof of Proposition~\ref{P:k2_general}, we have that, for any good prime $p$, $L_2$ permutes $\PP^1(\FF_p)$ if and only if the conjugacy class $\mathrm{Frob}_p$ is a $3$-cycle. Therefore, by the Chebotarev density theorem, the set of primes $p$ for which $L_2$ permutes $\PP^1(\FF_p)$ has natural density $1/3$ (resp. $2/3$) if $\Gal(L/\QQ)\simeq S_3$ (resp. $\Gal(L/\QQ)\simeq C_3$).
\item [(ii)] Since the nontrivial rational $2$-torsion points of $E/\QQ$ are exactly those with rational $x$-coordinate, Proposition~\ref{P:k2_general} confirms Odaba\c{s}'s conjecture for $k=2$. In fact, using Proposition~\ref{P:perm}, we can deduce that it holds for $k=2^n$, where $n\in \NN$.
\item [(iii)] Although Conjecture~\ref{C:Obadas} remains open for $k=3$, our computations provide supporting evidence: no potential counterexamples were found among elliptic curves of conductor $N\le 10^4$.
\end{itemize}
\end{remark}

To investigate the problem further for an arbitrary positive integer $k$, let us first point out additional counterexamples to Conjecture~\ref{C:Obadas} for \emph{non-CM} elliptic curves, which are of a similar nature to those discussed in Section~\ref{S:11}.
\begin{proposition}\label{P:counter_noncm}
    For any nonzero integer $u$, consider the elliptic curve 
\[E_u: y^2=x^3-9u^2x+12u^3.\]
Then we have the following.
\begin{itemize}
    \item [(i)] For any positive integer $k$ such that $6\mid k$ and any prime $p$ where $E_u$ has good reduction, $(A_p(E_u),k)\ne 1$; i.e., the $k$-th Latt\'{e}s map attached to $E_u$ is not arithmetically exceptional.
    \item [(ii)] If $P\in E_u(\overline{\QQ})$ is a non-trivial torsion point, then $x(P)\not\in \QQ$.
\end{itemize}
\end{proposition}
\begin{proof} 
    Let $k$ be a positive integer such that $6\mid k$.
    To prove (i), we shall consider the good primes $p$ in two different cases, depending on whether $3$ is a cubic residue modulo $p$. If $3\equiv m^3 \pmod p$ for some $m\in \ZZ$, then it can be checked easily that $(-mu-m^2u,0)$ is a $2$-torsion point on $E_u(\FF_p)$. Hence $|E_u(\FF_p)|$ must be even and, by \eqref{E:Np}, \[(A_p(E_u),k)=(|E_u(\FF_p)||E'_u(\FF_p)|,k)\ge 2.\] Now assume that $3$ is not a cubic residue modulo $p$. Then we have that $p\equiv 1 \pmod 3$. (Otherwise, for $p=3m+2$, we have $(3^{2m+1})^3\equiv 3^{2p-1} \equiv 3 \pmod p$.) Note that the $3$-rd division polynomial associated to $E_u$ is 
    \[\psi_3(x)=3(x^4-18u^2x^2+48u^3x
    -27u^4)=:3f(x),\]
    whose discriminant is $D=-2^{12}3^{13}u^{12}$. Since $\left(\frac{D}{p}\right)=\left(\frac{-3}{p}\right)=1$, it follows that $\psi_3(x)$ is reducible over $\FF_p$ (see, e.g., \cite{Dickson1906}). Suppose $f(x)$ is a product of two quadratic polynomials over $\FF_p$. Then we have that 
    \[f(x)=(x^2-Ax+B)(x^2+Ax+C),\]
    where $A\ne 0$. Expanding the right-hand side and comparing the coefficients yield $B+C=A^2-18u^2,$ $B-C=48u^3/A$, and $BC=-27u^4$. Since $(B+C)^2-(B-C)^2=4BC$, letting $a=A^2$ and $v=u^2$, one sees that $a(a-18v)^2-2304v^3=-108v^2a,$ which is equivalent to 
    \begin{equation}\label{E:av}
        \left(\frac{a-12v}{4v}\right)^3=9.
    \end{equation}
    Then it follows immediately from \eqref{E:av} that 
    \[\left(\frac{(a-12v)^2}{3(4v)^2}\right)^3=3,\]
    which contradicts the assumption on $p$. (Note that the multiplicative inverse of $3(4v)^2$ exists since $p$ does not divide the discriminant of $E_u$, which is $-2^6 3^5 u^6$.) Hence the only possibility is that $f(x)$ factors as the product of a linear and an irreducible cubic polynomial over $\FF_p$. In other words, $f(x)$ must have a unique root in $\FF_p$, say $x_0$. If $x_0^3-9u^2x_0+12u^3$ is a square in $\FF_p$, then $E_u$ has a $3$-torsion point over $\FF_p$. Otherwise, for any non-square $d$, $d(x_0^3-9u^2x_0+12u^3)$ must be a square in $\FF_p$ and the quadratic twist $E'_u:d^{-1}y^2=x^3-9u^2x+12u^3$ has a $3$-torsion point over $\FF_p$. Therefore, we have $(|E_u(\FF_p)||E'_u(\FF_p)|,k)\ge 3.$
    
     Next, assume to the contrary that there exists a non-trivial torsion point $P\in E_u(\overline{\QQ})$ such that $x(P)\in \QQ$ and let $n>1$ be the order of $P$.
     By the rational root theorem, both $t^3-9t+12$ and $t^4-18t^2+48t-27$ have no rational roots, so $E_u$ has no $2$- and $3$-torsion points with rational $x$-coordinates. Similarly, by investigating the $5$-th and $7$-th division polynomials of $E_u$, which have degree $12$ and $24$, respectively, it can be shown that $E_u$ has no $5$- and $7$-torsion points with rational $x$-coordinates. Hence we have that $q\nmid n$ for every $q\in \{2,3,5,7\}$. Let $\ell>7$ be a prime divisor of $n$ and let $T:=[n/\ell]P$. By the same arguments as in the proof of Proposition~\ref{P:D11_counterexample_curve}, we obtain that $T$ is a rational point of order $\ell$ on a quadratic twist of $E_u$, which contradicts Mazur's torsion theorem \cite{Mazur1977,Mazur1988}.
\end{proof}
\begin{remark}\label{R:counter10}
    A direct calculation yields $j(E_u)=-5184$, so $E_u$ does not have complex multiplication. We also discover many other non-CM counterexamples which are analogous to the family $E_u$. For example, for a nonzero integer $u$, consider the elliptic curve $F_u:y^2=x^3-60u^2x+180u^3$, which has $j$-invariant $-138240$. One can apply the same arguments as above to show that the statements in Proposition~\ref{P:counter_noncm} hold for $F_u$. In particular, for any prime $p\nmid 10u$, we have that $2\mid A_p(F_u)$ if $10$ is a cubic residue modulo $p$ and $3\mid A_p(F_u)$ otherwise. Furthermore, we conducted computations in \texttt{SageMath} searching for elliptic curves of conductor $N\le 10^3$ for which Conjecture~\ref{C:Obadas} might fail. In addition to the aforementioned counterexamples for $k=6$, we discovered several potential counterexamples for $k=5$ and $k=7$. For example, the curve $E_{11a2}$ (indexed using the Cremona label) does not have any $5$-torsion point with rational $x$-coordinate, but $5\mid A_p(E_{11a2})$ for all good primes $p<10^4$. Hence, by Proposition~\ref{P:perm}, the $5$-th Latt\'{e}s map attached to $E_{11a2}$ is \emph{unlikely} to be arithmetic exceptional. Similarly, the curve $E_{26b2}$ does not possess any $7$-torsion point whose $x$-coordinate is rational, while its $7$-th Latt\'{e}s map does not appear to be arithmetic exceptional. For each potential counterexample with $k\in\{5,7\}$, we found at least one curve in its isogeny class that \emph{does} contain a $k$-torsion point with rational $x$-coordinate. Nevertheless, this phenomenon does not occur for the elliptic curves $E_u$ defined in Proposition~\ref{P:D11_counterexample_curve} and Proposition~\ref{P:counter_noncm}, or for the elliptic curve $F_u$ defined above, when $6\mid k$. Based on these observations, it seems reasonable to reformulate Conjecture~\ref{C:Obadas} as follows.
    \end{remark} 
\begin{conjecture}\label{C:reform}
    For any positive integer $k$ such that $6\nmid k$, the $k$-th Latt\`es map attached to an elliptic curve $E/\mathbb{Q}$ is arithmetically exceptional if and only if no elliptic curve in the $\QQ$-isogeny class of $E$ possesses a $k$-torsion point whose $x$-coordinate is rational.
\end{conjecture}
As before, it is easy to prove that the forward direction of Conjecture~\ref{C:reform} holds for \emph{any} positive integer $k$. Let $E$ and $E'$ be elliptic curves in the same $\QQ$-isogeny class and suppose $E'$ has a $k$-torsion point whose $x$-coordinate is rational. Then it follows that the $k$-Latt\'{e}s map attached to $E'$ is not arithmetically exceptional. Since the Frobenius trace is invariant under an isogeny, we have $(A_p(E),k)=(A_p(E'),k)>1$ for all sufficiently large prime $p$, implying that the $k$-Latt\'{e}s map attached to $E$ is not arithmetic exceptional. Consequently, we immediately have from Theorem~\ref{T:main}, Proposition~\ref{P:D11_6notdivk_exc}, and Proposition~\ref{P:k2_general} that Conjecture~\ref{C:reform} is true for every CM elliptic curve $E/\QQ$ and it holds unconditionally for $k=2^n$, where $n\in \NN$. Our \texttt{SageMath} computations also provide further evidence; specifically, it appears to hold for every positive integer $k\le 15$, if $E$ has conductor $N\le 10^3$. While this conjecture remains open in full generality, we propose the following approach to investigate it via Galois representations.

Let $\ell$ be a prime and consider the mod-$\ell$ Galois representation
\[
\overline\rho_{E,\ell}:\Gal(\Qbar/\QQ)\longrightarrow \Aut(E[\ell])\simeq \GL_2(\FF_\ell).
\]
For a prime $p$ of good reduction with $p\neq \ell$, the element $\overline\rho_{E,\ell}(\mathrm{Frob}_p)$ has
characteristic polynomial congruent to $T^2-a_pT+p$ modulo $\ell$.
Evaluating at $T=\pm 1$ gives the congruence
\begin{equation*}
\det\bigl(I_2\pm \overline\rho_{E,\ell}(\mathrm{Frob}_p)\bigr)
\equiv p+1\pm a_p \pmod{\ell}.
\end{equation*}
Consequently, we have
\begin{equation}\label{E:eigenvalue_condition_rechecked}
 \bigl(A_p(E),\,m\bigr)=1
\ \Longleftrightarrow\
\text{for every }\ell\mid m,\ 
\det\bigl(I_2\pm\overline\rho_{E,\ell}(\mathrm{Frob}_p)\bigr)\not\equiv 0\pmod{\ell}.
\end{equation}
This viewpoint packages the permutation condition into a single Chebotarev-type condition.
Indeed, let
\[
\overline\rho_{E,m}:\Gal(\Qbar/\QQ)\longrightarrow \Aut(E[m])\simeq \GL_2(\ZZ/m\ZZ)
\]
be the combined mod-$m$ representation and set
\begin{align*}
G_m&:=\overline\rho_{E,m}\bigl(\Gal(\Qbar/\QQ)\bigr)\ \simeq\ \Gal\bigl(\QQ(E[m])/\QQ\bigr),\\
C_m&:=\Bigl\{A\in G_m:\ \det(I_2-A)\det(I_2+A)\in (\ZZ/m\ZZ)^\times\Bigr\}.
\end{align*}
For any prime $p\nmid m$ of good reduction, by \eqref{E:eigenvalue_condition_rechecked} and Proposition~\ref{P:perm}, we have
\[
L_m \text{ permutes }\PP^1(\FF_p)
\quad\Longleftrightarrow\quad
\overline\rho_{E,m}(\mathrm{Frob}_p)\in C_m.
\]
Since every element in $G_m$ comes from a Frobenius element, the Chebotarev density theorem implies the following simple restatement:
\begin{equation}\label{E:LmCm}
L_m \text{ is arithmetically exceptional if and only if }
C_m\neq \emptyset.
\end{equation}
Moreover, when $C_m\neq\emptyset$, the set of primes $p$ of good reduction for which $L_m$ permutes $\PP^1(\FF_p)$ has positive natural density and is equal to $|C_m|/|G_m|$.

For CM curves, the image groups $G_m$ admit an explicit description in terms of class field theory and reciprocity,
and the constructions in Sections~\ref{S:proof}--\ref{S:11} can be viewed as producing Frobenius elements landing in $C_m$.
The obstruction arising in the counterexamples presented above may be interpreted as the emptiness of $C_{k}$.  On the other hand, for any non-CM elliptic curve $E/\QQ$, we have from Serre's open image theorem \cite{Serre1972} that $\overline\rho_{E,\ell}$ is surjective for all sufficiently large primes $\ell$. This fundamental result can also be applied to obtain a partial solution to our problem for non-CM elliptic curves. Let $k>1$ be an integer and let $m=\rad(k)=\prod_{i=1}^r\ell_i$, where $\ell_i$'s are sufficiently large primes. Then we have 
\[E[m]\simeq \bigoplus_{i=1}^rE[\ell_i] \quad \text{and} \quad \overline\rho_{E,m}=\prod_{i=1}^r \overline\rho_{E,\ell_i}.\]
Then the Chinese remainder theorem gives
\[G_m=\prod_{i=1}^r G_{\ell_i}= \prod_{i=1}^r \GL_2(\ZZ/\ell_i\ZZ)\simeq \GL_2(\ZZ/m\ZZ),\]
i.e., $\overline\rho_{E,m}$ is surjective. Choose $a\in (\ZZ/m\ZZ)^\times$ such that $a\not\equiv \pm 1 \pmod {\ell_i}$ for all $1\le i \le r$ and let 
\[A=\begin{pmatrix}
    a & 0 \\
    0 & -a^{-1}
\end{pmatrix}.\]
Then it is clear that $A\in G_m$ and $\det(I_2\pm A)\not\equiv 0 \pmod{\ell_i}$, so $A\in C_m$. As a consequence of \eqref{E:LmCm}, $L_k$ is arithmetically exceptional, implying that Conjecture~\ref{C:reform} holds for any such $k$. Note that this approach is less effective than our proof for CM curves, since no uniform bound is currently known for the largest exceptional prime in Serre's open image theorem. Heuristic evidence suggests that the optimal uniform bound should be $37$ (see, e.g., \cite{Sutherland2016}). Motivated by the preceding investigation, we pose the following problem, which may be of independent interest.
\begin{problem}\label{Prob:nonCM_program}
Let $E/\QQ$ be an elliptic curve without CM and let $k\ge 2$, with $m=\rad(k)$.
Give an explicit criterion, in terms of the images of the representations $\rho_{E,m}$ that decides whether $C_m$ is empty.
\end{problem}
We conclude this paper by outlining two additional directions for future research. First, the notion of arithmetic exceptionality extends naturally to rational functions over a number field $K$, with $\FF_p$ replaced by the residue field $\FF_\mathfrak{p}$ at a prime ideal $\mathfrak{p}$ of $\mathcal{O}_K$. Therefore, it would be interesting to investigate these problems for elliptic curves defined over $K$. Second, as pointed out by one of the referees, one may also study an analogous problem for \emph{rigid} Latt\'{e}s maps arising from multiplication-by-$\alpha$ endormorphisms, where $\alpha$ is a non-integral complex number (cf. Section~\ref{S:aux}). In this setting, the situation appears to be more intricate, since one may need to formulate a new criterion for the arithmetic exceptionality of a rigid Latt\'{e}s map in terms of absence of certain points in $\Ker([\alpha])$, rather than torsion points. It would also be worthwhile to investigate this problem from the perspective of Galois representations.

\section*{Acknowledgments}
This research project is supported by King Mongkut’s University of Technology Thonburi (KMUTT), Thailand Science Research and Innovation (TSRI), and National Science, Research and Innovation Fund (NSRF) Fiscal year 2025 Grant number  FRB690020/0164. The authors would like to thank the referees for several valuable comments and suggestions which greatly help to improve the exposition of this paper. 

\bibliographystyle{amsplain}
\bibliography{ref}

\providecommand{\bysame}{\leavevmode\hbox to3em{\hrulefill}\thinspace}
\providecommand{\MR}{\relax\ifhmode\unskip\space\fi MR }
\providecommand{\MRhref}[2]{%
  \href{http://www.ams.org/mathscinet-getitem?mr=#1}{#2}
}
\providecommand{\href}[2]{#2}
\begin{thebibliography}{10}

\bibitem{BCCHLTZ22}
Z.~Bell, J.~Camero, K.~Cho, T.~Hyde, C.~M. Lu, B.~Thompson, and E.~Zhu,
  \emph{Density of periodic points for {L}att\`es maps over finite field}, J.
  Number Theory \textbf{238} (2022), 951--966. \MR{4430126}

\bibitem{BE2014}
J.~Bella\"{\i}che and N.~D. Elkies, \emph{Elliptic curves with trace of
  {F}robenius values always congruent to $0$ modulo $2$}, MathOverflow, 2014,
  mathoverflow.net/questions/179617.

\bibitem{BT2018}
G.~Bisson and M.~Tibouchi, \emph{Constructing permutation rational functions
  from isogenies}, SIAM J. Discrete Math. \textbf{32} (2018), no.~3,
  1741--1749. \MR{3829510}

\bibitem{Childress2009}
N.~Childress, \emph{Class field theory}, Universitext, Springer, New York,
  2009. \MR{2462595}

\bibitem{Cohen1993}
H.~Cohen, \emph{A course in computational algebraic number theory}, Graduate
  Texts in Mathematics, vol. 138, Springer-Verlag, Berlin, 1993. \MR{1228206}

\bibitem{Dickson1906}
L.~E. Dickson, \emph{Criteria for the irreducibility of functions in a finite
  field}, Bull. Amer. Math. Soc. \textbf{13} (1906), no.~1, 1--8. \MR{1558390}

\bibitem{Fried1970}
M.~Fried, \emph{On a conjecture of {S}chur}, Michigan Math. J. \textbf{17}
  (1970), 41--55. \MR{257033}

\bibitem{GMS2003}
R.~M. Guralnick, P.~M\"{u}ller, and J.~Saxl, \emph{The rational function
  analogue of a question of {S}chur and exceptionality of permutation
  representations}, Mem. Amer. Math. Soc. \textbf{162} (2003), no.~773,
  viii+79. \MR{1955160}

\bibitem{IR1990}
K.~Ireland and M.~Rosen, \emph{A classical introduction to modern number
  theory}, second ed., Graduate Texts in Mathematics, vol.~84, Springer-Verlag,
  New York, 1990. \MR{1070716}

\bibitem{JimenezUrrozANTS}
J.~Jim{\'e}nez~Urroz, \emph{Almost prime orders of {CM} elliptic curves modulo
  $p$}, Algorithmic Number Theory (ANTS--VIII) (Berlin), Lecture Notes in
  Computer Science, vol. 5011, Springer, 2008, pp.~74--87.

\bibitem{Kucuksakalli2014}
\"{O}. K\"{u}\c{c}\"{u}ksakall\i, \emph{Value sets of {L}att\`es maps over
  finite fields}, J. Number Theory \textbf{143} (2014), 262--278. \MR{3227347}

\bibitem{KO2018}
\"O. K\"u\c{c}\"uksakall{\i} and H.~\"Onsiper, \emph{Arithmetic exceptionality
  of generalized {L}att\`es maps}, J. Math. Soc. Japan \textbf{70} (2018),
  no.~2, 823--832.

\bibitem{Lemmermeyer2000}
F.~Lemmermeyer, \emph{Reciprocity laws: {F}rom {E}uler to {E}isenstein},
  Springer Monographs in Mathematics, Springer-Verlag, Berlin, 2000.
  \MR{1761696}

\bibitem{Mazur1977}
B.~Mazur, \emph{Modular curves and the {E}isenstein ideal}, Inst. Hautes
  \'{E}tudes Sci. Publ. Math. (1977), no.~47, 33--186 (1978), With an appendix
  by Mazur and M. Rapoport. \MR{488287}

\bibitem{Mazur1988}
\bysame, \emph{Rational isogenies of prime degree (with an appendix by {D}.
  {G}oldfeld)}, Invent. Math. \textbf{44} (1978), no.~2, 129--162. \MR{482230}

\bibitem{Milne2020}
J.S. Milne, \emph{Class field theory (v4.03)}, 2020, Available at
  www.jmilne.org/math/, pp.~287+viii.

\bibitem{Milnor2006}
J.~Milnor, \emph{On {L}att\`es maps}, Dynamics on the {R}iemann sphere, Eur.
  Math. Soc. Lect. Notes, European Mathematical Society, Z\"urich, 2006,
  pp.~9--43.

\bibitem{Muller1997}
P.~M\"{u}ller, \emph{A {W}eil-bound free proof of {S}chur's conjecture}, Finite
  Fields Appl. \textbf{3} (1997), no.~1, 25--32. \MR{1429041}

\bibitem{Neukirch1999}
J.~Neukirch, \emph{Algebraic number theory}, Grundlehren der mathematischen
  Wissenschaften [Fundamental Principles of Mathematical Sciences], vol. 322,
  Springer-Verlag, Berlin, 1999, Translated from the 1992 German original and
  with a note by Norbert Schappacher, With a foreword by G. Harder.
  \MR{1697859}

\bibitem{Odabas2023}
O.~Odaba{\c{s}}, \emph{Arithmetically exceptional {L}att{\`e}s maps attached to
  elliptic curves without complex multiplication}, M.{S}c. thesis, Middle East
  Technical University, Ankara, Turkey, June 2023.

\bibitem{Olson1974}
L.~D. Olson, \emph{Points of finite order on elliptic curves with complex
  multiplication}, Manuscripta Math. \textbf{14} (1974), 195--205. \MR{352104}

\bibitem{RS2009}
K.~Rubin and A.~Silverberg, \emph{Point counting on reductions of {CM} elliptic
  curves}, J. Number Theory \textbf{129} (2009), no.~12, 2903--2923.
  \MR{2560842}

\bibitem{Schur1923Uber}
I.~Schur, \emph{Über den zusammenhang zwischen einem problem der zahlentheorie
  und einem satz über algebraische funktionen}, S.-B. Preuss. Akad. Wiss.,
  Phys.-Math. Klasse (1923), 123--134.

\bibitem{Serre1972}
J.~P. Serre, \emph{Propri\'{e}t\'{e}s galoisiennes des points d'ordre fini des
  courbes elliptiques}, Invent. Math. \textbf{15} (1972), no.~4, 259--331.
  \MR{387283}

\bibitem{Silverman2009}
J.~H. Silverman, \emph{The arithmetic of elliptic curves}, second ed., Graduate
  Texts in Mathematics, vol. 106, Springer, Dordrecht, 2009. \MR{2514094}

\bibitem{SS2011}
J.~H. Silverman and K.~E. Stange, \emph{Amicable pairs and aliquot cycles for
  elliptic curves}, Exp. Math. \textbf{20} (2011), no.~3, 329--357.
  \MR{2836257}

\bibitem{SSP26}
S.~Sriwongsa, D.~Samart, and C.~Panraksa, \emph{Supplementary data for
  {A}rithmetic exceptionality of {L}att\`{e}s maps}, 2026, {\tt
  https://github.com/petesamart/Arithmetic-exceptionality-of-Lattes-maps}.

\bibitem{Sutherland2016}
A.~V. Sutherland, \emph{Computing images of {G}alois representations attached
  to elliptic curves}, Forum Math. Sigma \textbf{4} (2016), Paper No. e4, 79.
  \MR{3482279}

\bibitem{sagemath}
{The Sage Developers}, \emph{{S}agemath, the {S}age {M}athematics {S}oftware
  {S}ystem}, 2024, {\tt https://www.sagemath.org}.

\bibitem{Turnwald1995}
G.~Turnwald, \emph{On {S}chur's conjecture}, J. Austral. Math. Soc. Ser. A
  \textbf{58} (1995), no.~3, 312--357. \MR{1329867}

\end{thebibliography}
\end{document}